\documentclass[11pt,twoside,reqno]{amsart}
\usepackage{amssymb}
\usepackage{enumerate}
\usepackage{amsmath}
\usepackage{a4wide}
\usepackage[usenames]{color}
\usepackage[table, dvipsnames]{xcolor}
\usepackage{soul}
\usepackage[all]{xy}
\usepackage{array}
\usepackage{graphicx}
\usepackage{tabularx}
\usepackage{booktabs}
\usepackage{verbatim}
\usepackage[hidelinks=true]{hyperref} 

\usepackage{mathtools}

\usepackage{lmodern}
\usepackage[utf8]{inputenc}
\usepackage[L7x]{fontenc}  
\definecolor{blanchedalmond}{rgb}{1.0, 0.92, 0.8}

\usepackage{array}

\usepackage[
   backend=bibtex, 
   isbn=false,
   doi=false,
   url=false]{biblatex}
\newtheorem{theorem}{Theorem}

\newtheorem{lemma}[theorem]{Lemma}
\newtheorem{proposition}[theorem]{Proposition}

\theoremstyle{remark}

\newtheorem*{acknowledgement}{\textbf{Acknowledgments}}

\def\leq{\leqslant} \def\geq{\geqslant} \def\al{\alpha}

\begin{document}

\title{Linear relations of four conjugates of an algebraic number}

\author{Žygimantas Baronėnas, Paulius Drungilas, and Jonas Jankauskas}

\address{Institute of Mathematics, Faculty of Mathematics and Informatics, Vilnius
University, Naugarduko 24, Vilnius LT-03225, Lithuania}
\email{zygimantas.baronenas@mif.stud.vu.lt}

\address{Institute of Mathematics, Faculty of Mathematics and Informatics, Vilnius
University, Naugarduko 24, Vilnius LT-03225, Lithuania}
\email{paulius.drungilas@mif.vu.lt}

\address{Institute of Mathematics, Faculty of Mathematics and Informatics, Vilnius
University, Naugarduko 24, Vilnius LT-03225, Lithuania}
\email{jonas.jankauskas@mif.vu.lt}


\begin{abstract}
We characterize all algebraic numbers $\alpha$ of degree $d\in\{4,5,6,7\}$ for which there exist four distinct algebraic 
conjugates $\alpha_1$, $\alpha_2$, $\alpha_3$, $\alpha_4$ of $\alpha$ satisfying the relation $\alpha_{1}+\alpha_{2}=\alpha_{3}+\alpha_{4}$. 
In particular, we prove that an algebraic number $\alpha$ of degree 6  satisfies this relation with $\alpha_{1}+\alpha_{2}\notin\mathbb{Q}$ if and only if $\alpha$ is the sum of a quadratic and a cubic algebraic number. Moreover, we describe all possible Galois groups of the normal closure of $\mathbb{Q}(\alpha)$ for such algebraic numbers $\alpha$. We also consider similar relations $\alpha_{1}+\alpha_{2}+\alpha_{3}+\alpha_{4}=0$ and $\alpha_{1}+\alpha_{2}+\alpha_{3}=\alpha_{4}$ for algebraic numbers of degree up to 7. 
% We consider the linear equations $\alpha_{1}+\alpha_{2}+\alpha_{3}+\alpha_{4}=0, \alpha_{1}+\alpha_{2}+\alpha_{3}=\alpha_{4}$ and $\alpha_{1}+\alpha_{2}=\alpha_{3}+\alpha_{4}$ in conjugates of an algebraic number $\alpha$ of degree $d\leq 7$ over $\mathbb{Q}$. We prove that solutions for the first and third equations exist in the cases $d=4$ and $6$. Second equation, however, has no solutions. We also give explicit formulas for all possible minimal polynomials of such algebraic numbers. For $d=4$, the first equation is trivial (holds with any irreducible quartic polynomial with trace zero) and the third equation is solvable in roots of a quartic polynomial if and only if it is an irreducible polynomial of the form $p(x + c)$, where $p(x)$ - biquadratic polynomial and $c\in\mathbb{Q}$. For $d=6$, the first equation is solvable in roots of an irreducible sextic polynomial if and only if it is an irreducible polynomial of the form $x^{6}+ax^{4}+bx^{2}+ c \in\mathbb{Q}[x]$. The third equation required the most effort but we found a parametrization for all possible minimal polynomials and the Galois groups of the normal closure of $\mathbb{Q}(\alpha)$ over $\mathbb{Q}$. The proofs involve methods from linear algebra, Galois theory and some combinatorial arguments.
\end{abstract}
\maketitle

\section{Introduction}\label{intro}

Let $\alpha_{1}:=\alpha, \alpha_{2}, \dots, \alpha_{d}$ be the algebraic conjugates of an algebraic number $\alpha$ of degree $d$ over $\mathbb{Q}$. In the present paper we will be interested
in algebraic numbers $\alpha$ of small degree $d$ (namely, $d\leq 7$) whose conjugates satisfy one of the equations
\begin{align}\label{E:1}
\alpha_{1}+\alpha_{2}+\alpha_{3}+\alpha_{4}=0,\ \alpha_{1}+\alpha_{2}+\alpha_{3}=\alpha_{4}\ \text{ or }\ \alpha_{1}+\alpha_{2}=\alpha_{3}+\alpha_{4}.
\end{align}
%One may think of equations $\eqref{E:1}$ as the examples of simplest non-trivial linear additive relations that may occur among roots of irreducible integer polynomials.
The main motivation to study $\eqref{E:1}$ stems from the paper of Dubickas and Jankauskas \cite{DubickasJankauskas2015} where they investigated the linear relations $\alpha_{1}=\alpha_{2}+\alpha_{3}$ and $\alpha_{1}+\alpha_{2}+\alpha_{3}=0$ in conjugates of an algebraic number $\alpha$ of degree $d\leq 8$ over $\mathbb{Q}$. They proved that solutions to those equations exist only in the case $d = 6$ (except for the trivial solution of the second equation in cubic numbers with trace zero) and gave explicit formulas for all possible minimal polynomials of such algebraic numbers. In particular, equation $\alpha_{1} = \alpha_{2} + \alpha_{3}$ is solvable in roots of an irreducible sextic polynomial if and only
if it is an irreducible polynomial of the form 
\[
p(x)=x^{6}+2ax^{4}+a^{2}x^{2}+b\in\mathbb{Q}[x]. 
\]
Similarly, for $d$ in the range $4\leq d\leq 8$ (the case $d=3$ is trivial), equation $\alpha_{1} + \alpha_{2} + \alpha_{3} = 0$ is solvable if and only if $d = 6$ and the minimal polynomial of $\alpha$ over $\mathbb{Q}$ is an irreducible polynomial of the form 
\[
p(x) = x^{6} + 2ax^{4} + 2bx^{3} + (a^{2} - c^{2}t)x^{2} + 2(ab - cet)x + b^{2} - e^{2}t
\]
for some rational numbers $a, b, c, e\in\mathbb{Q}$ and some square-free integer $t\in\mathbb{Z}$.

Let $\alpha_1,\alpha_2,\alpha_3$ be three distinct algebraic conjugates of an algebraic number $\alpha$ of degree $d\leq 8$. Recently, Virbalas \cite{Virbalas2025} extended the research of Dubickas and Jankauskas \cite{DubickasJankauskas2015} by determining all possible linear relations of the form 
$a\alpha_1+b\alpha_2+c\alpha_3=0$ with non-zero rational numbers 
$a, b, c$. He also obtained a complete list of transitive groups that can
occur as Galois groups for the minimal polynomial of such an algebraic number $\alpha$. Moreover, Virbalas \cite{Virbalas2025a} proved that for any prime number $p\geq 5$ there does not exist an irreducible polynomial $p(x)\in\mathbb{Q}[x]$ 
of degree $2p$ whose three distinct roots sum up to zero.

Recently, Dubickas and Virbalas \cite{DuVi25} proved that every nontrivial linear relation between algebraic conjugates has a corresponding multiplicative relation. They also gave a complete characterization of all 
possible linear relations between four distinct algebraic conjugates of degree 4 (see, also, \cite{Kitaoka2017}).

Recall that a real algebraic integer $\alpha>1$ is called a Pisot number if all of its conjugates $\alpha_{j}$, other than $\alpha$ itself, satisfy $|\alpha_{j}|<1$. Dubickas, Hare, and Jankauskas in \cite{DubickasHareJankauskas2017} showed that there are no Pisot numbers whose conjugates satisfy the equation $\alpha_{1}=\alpha_{2}+\alpha_{3}$. They also proved the impossibility of 
\begin{align}\label{E:2}
\alpha_{1}+\alpha_{2}=\alpha_{3}+\alpha_{4}
\end{align}
 in conjugates of a Pisot number of degree $d>4$, by showing that there is a unique Pisot number, namely, $\alpha = (1+\sqrt{3+2\sqrt{5}})/2$ whose conjugates satisfy $\eqref{E:2}$. This particular number $\alpha$ was first found in \cite{DubickasSmyth2008}.

Throughout this paper, the term \textit{algebraic number} means algebraic number over the field of rational numbers $\mathbb{Q}$. Similarly, the term \textit{irreducible polynomial} means irreducible over $\mathbb{Q}$. Let $\alpha_1=\alpha, \alpha_2, \dotsc, \alpha_d$ be the algebraic conjugates of an algebraic number $\alpha$ of degree $d$. Then $tr(\alpha):=\alpha_1+\alpha_2+\dotsb+\alpha_d$ is called \textit{the trace} (or \textit{the absolute trace}) of $\alpha$. In the present paper, we restrict ourselves to the degrees in the range $4\leq d\leq 7$. The case $d=8$ is more complicated and will be treated in the future. We will not assume that $\alpha$ is a Pisot number in the equations $\eqref{E:1}$. For the first two equations in $\eqref{E:1}$ we have the following result:
\begin{theorem}\label{111}
Let $\alpha$ be an algebraic number of degree $d$, where $d\in\{4,5,6,7\}$.
\begin{itemize}
\item[(i)] Some four distinct algebraic conjugates of $\alpha$ satisfy the relation%Then some four of its conjugates $\alpha_{1}, \alpha_{2}, \alpha_{3}, \alpha_{4}$ satisfy the relation
\[
\alpha_{1}+\alpha_{2}+\alpha_{3}+\alpha_{4}=0
\]
if and only if $d=4$ and $tr(\alpha)=0$  or $d=6$ and the minimal polynomial of $\alpha$ is an irreducible polynomial of the form
\[
x^{6}+ax^{4}+bx^{2}+ c \in\mathbb{Q}[x].
\]
\item[(ii)] No four distinct conjugates of $\alpha$ satisfy the relation
\[
\alpha_{1}+\alpha_{2}+\alpha_{3}=\alpha_{4}.
\]
\end{itemize}
\end{theorem}
%\begin{theorem}\label{111}
%Let $d$ be an integer in the range $4\leq d\leq 7$ and let $\alpha$ be an algebraic number of degree $d$ over $\mathbb{Q}$. Then some four distinct conjugates of $\alpha$ satisfy the relation%Then some four of its conjugates $\alpha_{1}, \alpha_{2}, \alpha_{3}, \alpha_{4}$ satisfy the relation
%\[
%\alpha_{1}+\alpha_{2}+\alpha_{3}+\alpha_{4}=0
%\]
%if and only if $d=4$ or $6$. The case $d=4$ holds with any irreducible quartic polynomial with trace zero. In the case $d=6$, the minimal polynomial of $\alpha$ over $\mathbb{Q}$ is an irreducible polynomial of the form
%\[
%x^{6}+ax^{4}+bx^{2}+ c \in\mathbb{Q}[x].
%\]
%\end{theorem}
%The second equation in $\eqref{E:1}$ has no solutions for the degrees in the range $4\leq d\leq 7$:
%\begin{theorem}\label{222}
%Let $d$ be an integer in the range $4\leq d\leq 7$ and let $\alpha$ be an algebraic number of degree $d$ over $\mathbb{Q}$. Then none of its conjugates satisfy the relation
%\[
%\alpha_{1}+\alpha_{2}+\alpha_{3}=\alpha_{4}.
%\]
%\end{theorem}

%The third equation in $\eqref{E:1}$ for $d=6$ is the most interesting case. 

% The main results regarding the possible cases and minimal polynomials for the third equation in $\eqref{E:1}$ for $4\leq d\leq 7$ are given in the following theorem:

The following theorem treats algebraic numbers $\alpha$ of degree $d\in\{4,5,6,7\}$  whose some four distinct algebraic conjugates satisfy the relation
\begin{equation}\label{eq:dd}
\alpha_{1}+\alpha_{2}=\alpha_{3}+\alpha_{4}.
\end{equation}
Note that for any $r\in\mathbb{Q}$, replacing $\alpha$ with $\alpha-r$ does not affect the relation \eqref{eq:dd}. By setting $r=tr(\alpha)/d$ we obtain an algebraic number $\alpha-tr(\alpha)/d$ whose trace equals zero. Therefore, without the loss of generality, we assume that $\alpha$ has zero trace.  

\begin{theorem}\label{444}
Let $\alpha$ be an algebraic number of degree $d\in\{4,5,6,7\}$  and $tr(\alpha)=0$. 
Denote by $p(x)$ the minimal polynomial of $\alpha$. 
Suppose that some four distinct algebraic conjugates of $\alpha$ satisfy the relation \eqref{eq:dd}. 
Then either $d=4$ or $d=6$.
Moreover, the following statements are true.  
\begin{itemize}
\item[$(i)$] If $d=4$, then $p(x)$ is an irreducible polynomial of the form
\[
p(x) = x^{4}+ax^{2}+b
\]
for some rational numbers $a, b\in\mathbb{Q}$. Conversely, for any such irreducible polynomial $p(x)$, some four distinct roots of $p(x)$ satisfy the relation \eqref{eq:dd}. 

\item[$(ii)$] Suppose that $d=6$ and the sum $\alpha_{1}+\alpha_{2}$ in \eqref{eq:dd} is a rational number.  Then $p(x)$ is an irreducible polynomial of the form
\[
p(x)=x^{6}+ax^{4}+bx^{2}+ c
\]
for some rational numbers $a, b, c\in\mathbb{Q}$. Conversely, for any such irreducible polynomial $p(x)$, some four distinct roots of $p(x)$ satisfy the relation \eqref{eq:dd}. 

\item[$(iii)$] Suppose that $d=6$ and the sum $\alpha_{1}+\alpha_{2}$ in \eqref{eq:dd} is not a rational number (i.e., $\alpha_{1}+\alpha_{2}\in\mathbb{C}\setminus\mathbb{Q}$).  
Then $p(x)$ is an irreducible polynomial of the form 
% \begin{equation*}
% \begin{split}
% p(x)&=x^{6} - (3a-2b)x^{4} - 2cx^{3} + (3a^{2} + b^{2})x^{2} - 2(3a + b)cx\\ & - (a^{3} + 2a^{2}b + ab^{2} - c^{2})
% \end{split}
% \end{equation*}
\begin{equation*}
\begin{split}
p(x)=&x^{6} + (2 b - 3 a) x^{4} + 2 c x^{3} + (3 a^{2} + b^{2}) x^{2} + 2c(3 a + b) x\\ &- a^{3} - 2 a^{2} b - a b^{2} + c^{2}.
\end{split}
\end{equation*}
for some rational numbers $a, b, c\in\mathbb{Q}$. Conversely, for any such irreducible polynomial $p(x)$, some four distinct roots of $p(x)$ satisfy the relation \eqref{eq:dd}.  
% Simultaneously, the minimal polynomial of $\beta$ over $\mathbb{Q}$ is an irreducible polynomial of the form $R(x+2e)$, where
% \[
% R(x) = x^{3} + bx + c.
% \]
\end{itemize}
% {\color{blue}
% Gali buti tokia (iii) punkto alternatyva:
% \begin{itemize}
% \item[(iii')] Suppose that $d=6$ and  $\alpha_{1}+\alpha_{2}=\alpha_{3}+\alpha_{4}=:\beta\notin\mathbb{Q}$. Then $\alpha$ equals the sum of a quadratic and a cubic algebraic number. Conversely, if $\alpha$ equals the sum of a quadratic and a cubic algebraic number then some four distinct conjugates of $\alpha$ satisfy the relation \eqref{eq:dd} and $\alpha_{1}+\alpha_{2}=\alpha_{3}+\alpha_{4}=:\beta\notin\mathbb{Q}$.
% \end{itemize}
% Pasirinkus (iii') atskirai galima butu suformuluoti teigini: $\alpha$ yra kvadratinio ir kubinio suma tada ir tik tada, kai $\alpha$ minimalusis polinomas turi format, nurodyta (iii) punkte.
% }
\end{theorem}

The following theorem gives an alternative description of sextic algebraic numbers $\alpha$ that satisfy the relation \eqref{eq:dd} with $\alpha_{1}+\alpha_{2}\notin\mathbb{Q}$. We will derive this result from Proposition~\ref{prop562} (see Section~\ref{intro2}).

\begin{theorem}\label{propqc}
Let $\alpha$ be an algebraic number of degree 6. Some four distinct algebraic conjugates of $\alpha$ satisfy the relation $\alpha_{1}+\alpha_{2}=\alpha_{3}+\alpha_{4}=:\beta\notin\mathbb{Q}$ 
if and only if $\alpha$ equals the sum of a quadratic and a cubic algebraic number.
\end{theorem}

Let $\alpha$ be an algebraic number of degree $d$ and let $G$ be the Galois group of the normal closure of $\mathbb{Q}(\alpha)$ over $\mathbb{Q}$. The group $G$ is determined (in a unique way) by its action on $\mathcal{S}=\{\alpha_{1}, \alpha_{2},\dots, \alpha_{d}\}$: it corresponds to some transitive subgroup of the full symmetric group $S_{d}$. Next, we will consider possible groups $G$, related to algebraic numbers $\alpha$ in Theorem \ref{111} and Theorem \ref{444}. 

If $d=4$ and the linear relation in Theorem \ref{111} is satisfied, then $G$ is isomorphic to one of 5 transitive subgroups of the symmetric group $S_{4}$, namely, $V_{4}$ (Klein 4-group), $C_{4}$ (a cyclic group of order 4), $D_{4}$ (a dihedral group of order 8), $A_{4}$ (the alternating group) or $S_{4}$ itself. There are no more transitive subgroups of $S_{4}$ (see, e.g., \cite[Chapter 3]{dixon1996} or \cite{rotman1995}).

If $d=6$ and some four distinct conjugates of $\alpha$ satisfy the relation in $(i)$ of Theorem \ref{111}, then we need to look at the transitive subgroups of $S_{6}$. Awtrey and Jakes in \cite{AwtreyJakes2020} investigated the Galois groups of even sextic polynomials $x^{6} + ax^{4} + bx^{2} + c$ with coefficients from a field of characteristic $\neq 2$. In this particular case, there are 8 possibilities for the Galois group $G$:
\begin{equation}\label{E:3}
C_{6},\ S_{3},\ D_{6},\ A_{4},\ A_{4}\times C_{2},\ S_{4}^{+},\ S_{4}^{-},\ S_{4}\times C_{2},
\end{equation}
where $S_{4}^{+}$ and $S_{4}^{-}$ are certain transitive subgroups of $S_{6}$ of order $24$. %In particular, we can describe these subgroups using generators:
%\begin{equation*}
%\begin{split}
%S_{4}^{+} &= \langle (1\ 4)(2\ 5), (1\ 3\ 5)(2\ 4\ 6), (1\ 5)(2\ 4) \rangle,\\
%S_{4}^{-} &= \langle (1\ 4)(2\ 5), (1\ 3\ 5)(2\ 4\ 6), (1\ 5)(2\ 4)(3\ 6) \rangle.
%\end{split}
%\end{equation*}
Note that, in total, there are 16 transitive subgroups of $S_{6}$ (see, e.g., \cite[Chapter 3]{dixon1996}). Awtrey and Jakes in \cite{AwtreyJakes2020} also provided one-parameter families of even sextic polynomials (for values of $t\in\mathbb{Q}$ that result in irreducible polynomials) with specified Galois group over $\mathbb{Q}$ (see Table~\ref{table:ce3}).

\begin{table}[h!]
\centering
\begin{tabular}{ |c|c| } 
\hline
Polynomial $p(x)$ & Galois group $G$ of $p(x)$\\
\hline
$x^{6} + (t^{2}+5)x^{4} + ((t-1)^{2}+5)x^{2} + 1$ & $C_{6}$\\ 
$x^{6} + 3t^{2}$ & $S_{3}$\\ 
$x^{6} + 2t^{2}$ & $D_{6}$\\
$x^{6} - 3t^{4}x^{2} - t^{6}$ & $A_{4}$\\
$x^{6} - 3t^{2}x^{2} + t^{3}$ & $A_{4}\times C_{2}$\\ 
$x^{6} + t^{2}x^{4} - t^{6}$ & $S_{4}^{+}$\\ 
$x^{6} + (31t^{2})^{2}x^{2} + (31t^{2})^{3}$ & $S_{4}^{-}$\\
$x^{6} + (2t^{2})^{2}x^{2} + (2t^{2})^{3}$ & $S_{4}\times C_{2}$\\
\hline
\end{tabular}
\vspace{0,2cm}
\caption{One-parameter families of even sextic polynomials $p(x)$ with corresponding Galois groups $G$.}
\label{table:ce3}
\end{table}

%\begin{equation*}
%\begin{split}
%C_{6}:\quad &x^{6} + (t^{2}+5)x^{4} + ((t-1)^{2}+5)x^{2} + 1; \\
%S_{3}:\quad &x^{6} + 3t^{2};\\
%D_{6}:\quad &x^{6} + 2t^{2};\\
%A_{4}:\quad &x^{6} - 3t^{4}x^{2} - t^{6};\\ 
%A_{4}\times C_{2}:\quad &x^{6} - 3t^{2}x^{2} + t^{3};\\ 
%S_{4}^{+}:\quad &x^{6} + t^{2}x^{4} - t^{6};\\ 
%S_{4}^{-}:\quad &x^{6} + (31t^{2})^{2}x^{2} + (31t^{2})^{3};\\ 
%S_{4}\times C_{2}:\quad &x^{6} + (2t^{2})^{2}x^{2} + (2t^{2})^{3}.
%\end{split}
%\end{equation*}
%\begin{center}
%\centering
%\begin{tabular}{ | m{1cm} | m{10cm}| } 
%  \hline
%  \center{\textbf{G}} & \textbf{Polynomials} \\ 
%  \hline
%  $C_{6}$ & $x^{6} + (t^{2}+5)x^{4} + ((t-1)^{2}+5)x^{2} + 1,\ t\in\mathbb{Z}$ \\ 
%  \hline
%  $S_{3}$ & $x^{6}+3t^{2}$ \\ 
%  \hline
%\end{tabular}
%\end{center}

If $d=4$ and the linear relation in \eqref{eq:dd} is satisfied, then $G$ is one of 3 transitive subgroups of the symmetric group $S_{4}$: $V_{4}, C_{4}$ or $D_{4}$. This result is due to Kappe and Warren (see Theorem~3 in \cite{KappeWarren1989}).  Again, Awtrey and Jakes in \cite{AwtreyJakes2020} provided one-parameter families of even quartic polynomials (except for values of $t\in\mathbb{Q}$ that result in reducible polynomials) with specified Galois group over $\mathbb{Q}$ (see Table~\ref{table:ce4}).

\begin{table}[h!]
\centering
\begin{tabular}{ |c|c| } 
\hline
Polynomial $p(x)$ & Galois group $G$ of $p(x)$\\
\hline
$x^{4} + (2t+1)^{2}$ & $V_{4}$\\ 
$x^{4} + 4tx^{2} + 2t^{2}$ & $C_{4}$\\ 
$x^{4} + t^{2} + 1,\ t\neq 0$ & $D_{4}$\\
\hline
\end{tabular}
\vspace{0,2cm}
\caption{One-parameter families of even quartic polynomials $p(x)$ with corresponding Galois groups $G$.}
\label{table:ce4}
\end{table}

%\begin{equation*}
%\begin{split}
%V_{4}:\quad &x^{4} + (2t+1)^{2}; \\
%C_{4}:\quad &x^{4} + 4tx^{2} + 2t^{2};\\
%D_{4}:\quad &x^{4} + t^{2} + 1,\ t\neq 0.
%\end{split}
%\end{equation*}

If $d=6$ and some four distinct conjugates of $\alpha$ satisfy $\alpha_{1}+\alpha_{2}=\alpha_{3}+\alpha_{4}\in\mathbb{Q}$, then the Galois group is, again, one of the already mentioned 8 transitive subgroups in $\eqref{E:3}$.%of $S_{6}$: $C_{6}, S_{3}, D_{6}, A_{4}, A_{4}\times C_{2}, S_{4}^{+}, S_{4}^{-}, S_{4}\times C_{2}$.

% Denote the permutations $\sigma=(1\ 3\ 5)(2\ 6\ 4),\ \tau=(1\ 2)(3\ 4)(5\ 6)$, and $\pi = (1\ 2\ 5\ 4\ 3\ 6)$ that may be applied to the set $\mathcal{S}$. 

The most interesting case is the following

\begin{theorem}\label{555}
Let $\alpha$ be an algebraic number of degree $6$ and $tr(\alpha)=0$. Suppose that some four distinct algebraic conjugates of $\alpha$ satisfy the relation
\[
\alpha_{1}+\alpha_{2}=\alpha_{3}+\alpha_{4}=:\beta\not\in\mathbb{Q}.
\]
Then the Galois group of the normal closure of $\mathbb{Q}(\alpha)$ over $\mathbb{Q}$ is isomorphic to one of the three groups: dihedral group $D_6$ of order 12, symmetric group $S_3$ and the cyclic group $C_6$.

% Then $\beta$ is a cubic algebraic number and one can label the algebraic conjugates $\alpha_1,\alpha_2,\dotsc,\alpha_6$ of $\alpha$ in such a way that these satisfy the relations
% \begin{equation}\label{eq:rel2}
% \begin{cases}
% \beta_{1} = \alpha_{1}+\alpha_{2} = \alpha_{3}+\alpha_{4},\\
% \beta_{2} = \alpha_{2}+\alpha_{5} = \alpha_{3}+\alpha_{6},\\
% \beta_{3} = \alpha_{1}+\alpha_{6} = \alpha_{4}+\alpha_{5},\\
% \end{cases}
% \end{equation}
% % \begin{equation}\label{eq:rel1}
% % \begin{cases}
% % -2\beta_{1} = \alpha_{5}+\alpha_{6},\\
% % -2\beta_{2} = \alpha_{1}+\alpha_{4},\\
% % -2\beta_{3} = \alpha_{2}+\alpha_{3},\\
% % \end{cases}
% % \end{equation}
% where $\beta_1=\beta,\beta_2,\beta_3$ are the algebraic conjugates of $\beta$. 
% Denote $\pi := (1\ 2\ 5\ 4\ 3\ 6)$, $\sigma= \pi^4= (1\ 3\ 5)(2\ 6\ 4)$ and $\tau=(1\ 2)(3\ 4)(5\ 6)$ permutations of the symmetric 
% group $S_6$. Let $G$ be the Galois group of the normal closure of $\mathbb{Q}(\alpha)$ over $\mathbb{Q}$. Assume that $G$ is a subgroup of $S_6$, acting on the indices of the conjugates $\alpha_1,\alpha_2,\dotsc,\alpha_6$ of $\alpha$. Then, given the relations \eqref{eq:rel2}, there are exactly three possible cases:
% \begin{enumerate}
% \item $G=\langle \tau, \pi\ |\ \tau^{2}=\pi^{6}=id,\ \tau\pi\tau = \pi^{5}\rangle\cong D_{6}$; %(a dihedral group of order 12);
% \item $G=\langle \pi\ |\ \pi^{6}=id\rangle\cong C_{6}$;
% \item $G = \{id, \sigma, \sigma^{2}, \tau, \tau\sigma, \tau\sigma^{2}\}\cong S_3$.
% \end{enumerate}
\end{theorem}

Theorem~\ref{555} follows from Proposition~\ref{prop562}, which gives more details on the possible Galois group of the normal closure of $\mathbb{Q}(\alpha)$. 
Moreover, all three groups in Theorem~\ref{555} arise as Galois groups in this setting, i.e., 
for any group $G\in\{D_6,S_3, C_6\}$ there exists an algebraic number $\alpha$ of degree $6$ satisfying $\alpha_{1}+\alpha_{2}=\alpha_{3}+\alpha_{4}\not\in\mathbb{Q}$ such that the Galois group of the normal closure of $\mathbb{Q}(\alpha)$ over $\mathbb{Q}$ is isomorphic to $G$. Corresponding examples are provided in Table~\ref{table:ce1}.

\begin{table}[h!]
\centering
\begin{tabular}{ |c|c|c| } 
\hline
Polynomial $p(x)$ & $(a,b,c)$ & Galois group $G$ of $p(x)$\\
\hline
 $x^{6} - 6 x^{4} + 4 x^{3} + 12 x^{2} + 24 x - 4$ & $(2,0,2)$ & $D_{6}$\\ 
 $x^{6} - 3 x^{4} + 2 x^{3} + 12 x^{2} - 12 x + 17$ & $(-1,-3,1)$ & $C_{6}$\\ 
 $x^{6} - 3 x^{4} + 8 x^{3} + 12 x^{2} - 48 x + 32$ & $(-1,-3,4)$ & $S_{3}$\\
\hline
\end{tabular}
\vspace{0,2cm}
\caption{Minimal polynomials $p(x)$ from part $(iii)$ of Theorem~\ref{444} with the corresponding Galois groups $G$.}
\label{table:ce1}
\end{table}

The converse of  Theorem~\ref{555} is false, i.e., for any group $G\in\{D_6,S_3, C_6\}$ there exists an algebraic number $\alpha$ of degree $6$ such that the Galois group of the normal closure of $\mathbb{Q}(\alpha)$ over $\mathbb{Q}$ is isomorphic to $G$ and no four distinct algebraic conjugates of $\alpha$ satisfy the relation $\alpha_{1}+\alpha_{2}=\alpha_{3}+\alpha_{4}$. Indeed, it suffices to take an irreducible polynomial of degree 6, having the specified Galois group, which is not of the form given in $(iii)$ of Theorem~\ref{444}. Such examples are provided in Table~\ref{table:ce2}. 

\begin{table}[h!]
\centering
\begin{tabular}{ |c|c| } 
\hline
Polynomial $p(x)$ & Galois group $G$ of $p(x)$\\
\hline
 $x^{6} + 2x^{3} + 2$ & $D_{6}$\\ 
 $x^{6} + x^{3} + 1$ & $C_{6}$\\ 
 $x^{6} + 54x^{3} + 1029$ & $S_{3}$\\
\hline
\end{tabular}
\vspace{0,2cm}
\caption{Polynomials $p(x)$ that are not of the form given in $(iii)$ of Theorem~\ref{444} with the corresponding Galois groups $G$.}
\label{table:ce2}
\end{table}

%\begin{example}
%Recall the parametrization
%\begin{equation*}
%x^{6} - 2(6a-b)x^{4} - 2cx^{3} + (48a^{2} + b^{2})x^{2} - 2(12a + b)cx - 64a^{3} - 32a^{2}b - 4ab^{2} + c^{2}.
%\end{equation*}
%\begin{enumerate}
%\item $D_{6}:\quad x^{6} + 6x^{4} - 6x^{3} + 57x^{2} + 90x + 205\quad (a=-1,\ b=-3,\ c=3)$;
%\item $C_{6}:\quad x^{6} + 6x^{4} - 2x^{3} + 57x^{2} + 30x + 197\quad (a=-1,\ b=-3,\ c=1)$;
%\item $S_{3}:\quad x^{6} + 6x^{4} - 8x^{3} + 57x^{2} + 120x + 212\quad (a=-1,\ b=-3,\ c=4)$.
%\end{enumerate}
%\end{example}

The paper is organized as follows. Auxiliary results are stated in Section~\ref{intro2}.  The proofs of the main results are given in Section~\ref{proofs}. We first prove Proposition~\ref{333} and \ref{prop562}. Theorem~\ref{555} directly follows from Proposition~\ref{prop562}. Then we use Proposition~\ref{prop562} to prove   Theorem~\ref{propqc}, which then is used to prove Theorem~\ref{444}.  

% In the next section we give some auxiliary results. Then, in Section 3, we prove Theorem \ref{111}. In Section 4, we obtain the results from Proposition \ref{333}. Next, in Section 5, we prove Theorem \ref{555} and finally, in Section 6, we derive the parametrizations from Theorem \ref{444}.

%In the next section we give some auxiliary results. Then, in Section 3, we prove Theorem \ref{111}. In Section 4, we obtain the results from Theorem \ref{444} for $d=4, 5$, and $7$ together with the claim from Proposition \ref{333}. Next, in Section 5, we prove Theorem \ref{555} and finally, in Section 6, we derive the parametrizations from Theorem \ref{444} for $d=6$.

\section{Auxiliary results}\label{intro2}

The following result is due to Kurbatov \cite{Kurbatov1977}. We will use it to eliminate impossible relations among algebraic conjugates.

\begin{lemma}\label{intro4}
The equality
\[
k_{1}\alpha_{1}+k_{2}\alpha_{2}+\cdots+k_{d}\alpha_{d} = 0
\]
with conjugates $\alpha_{1}, \alpha_{2},\dots, \alpha_{d}$ of an algebraic number $\alpha$ of prime degree $d$ over $\mathbb{Q}$ and $k_{1}, k_{2},\dots, k_{d}\in\mathbb{Z}$ can only hold if  $k_{1} = k_{2} = \cdots = k_{d}$.
\end{lemma}

Smyth's result from \cite{Smyth1982} is useful for similar purposes.

\begin{lemma}\label{intro7}
If $\alpha_{1}, \alpha_{2}, \alpha_{3}$ are three conjugates of an algebraic number satisfying $\alpha_{1}\neq\alpha_{2}$ then $2\alpha_{1}\neq \alpha_{2} + \alpha_{3}$.
\end{lemma}

The following result is a generalization of Lemma $\ref{intro7}$ proved by Dubickas \cite{Dubickas2002}.

\begin{lemma}\label{lemma4}
If $\beta_{1}, \beta_{2},\dots , \beta_{n}$, where $n\geq 3$, are distinct algebraic numbers conjugate over a field of characteristic zero $K$ and $k_{1}, k_{2},\dots ,k_{n}$ are non-zero rational numbers satisfying $|k_{1}| \geq |k_{2}|+\dots+|k_{n}|$ then
\[
k_{1}\beta_{1} + k_{2}\beta_{2} +\dots+ k_{n}\beta_{n}\notin K.
\]
\end{lemma}

Dubickas and Jankauskas in their paper \cite{DubickasJankauskas2015} proved the following result

\begin{lemma}\label{intro3}
The equality
\[
k_{1}\alpha_{1}+k_{2}\alpha_{2}+\cdots+k_{d}\alpha_{d} = 0
\]
with conjugates $\alpha_{1}, \alpha_{2},\dots, \alpha_{d}$ of an algebraic number $\alpha$ of degree $d$ over $\mathbb{Q}$ and $k_{1}, k_{2},\dots, k_{d}\in\mathbb{Z}$ satisfying $\sum_{i=1}^{d} k_{i}\neq 0$ can only hold if $tr(\alpha) := \alpha_{1} + \alpha_{2} + \cdots + \alpha_{d} = 0$.
\end{lemma}

The following result is a partial case of Theorem~1.3 in \cite{Weintraub2011}.

\begin{lemma}\label{lempe}
Suppose that $\alpha$ and $\beta$ are algebraic numbers over $\mathbb{Q}$ of degree $m$ and $n$, respectively. If $m$ and $n$ are coprime integers, then $\alpha+\beta$ is a primitive element of the compositum $\mathbb{Q}(\alpha,\beta)$, i.e., $\mathbb{Q}(\alpha,\beta)=\mathbb{Q}(\alpha+\beta)$. 
\end{lemma}

To prove Proposition~\ref{prop562} and Theorem~\ref{444} we will need the following result.
\begin{proposition}\label{333}
Let $\alpha$ be an algebraic number of degree $d=6$  and $tr(\alpha)=0$. Suppose that some four distinct conjugates of $\alpha$ satisfy the relation
\[
\alpha_{1}+\alpha_{2}=\alpha_{3}+\alpha_{4}.
\]
Then either $\alpha_{1}+\alpha_{2}=0$ or $\alpha_{1}+\alpha_{2}$ is an algebraic number of degree $3$ and $tr(\alpha_{1}+\alpha_{2})=0$.
\end{proposition}
%\begin{proposition}\label{333}
%Let $\alpha$ be an algebraic number of degree $d=6$ over $\mathbb{Q}$ with $tr(\alpha):=\alpha_{1}+\alpha_{2}+\alpha_{3}+\alpha_{4}+\alpha_{5}+\alpha_{6}=0$. Then some four distinct conjugates of $\alpha$ satisfy the relation
%\[
%\alpha_{1}+\alpha_{2}=\alpha_{3}+\alpha_{4}=\beta
%\]
%if and only if $\beta=0$ or $\beta$ -- algebraic number of degree $3$ over $\mathbb{Q}$ with $tr(\beta):=\beta_{1}+\beta_{2}+\beta_{3}=0$.
%\end{proposition}

Denote $\pi := (1\ 2\ 5\ 4\ 3\ 6)$, $\sigma= \pi^4= (1\ 3\ 5)(2\ 6\ 4)$ and $\tau=(1\ 2)(3\ 4)(5\ 6)$ 
permutations of the symmetric group $S_6$. 
Theorem~\ref{555} is a corollary of the following proposition, which will also be used in the proof of Theorem~\ref{propqc}. 

\begin{proposition}\label{prop562}
Let $\alpha$ be an algebraic number of degree $6$ and $tr(\alpha)=0$. Suppose that some four distinct algebraic conjugates of $\alpha$ satisfy the relation
\[
\alpha_{1}+\alpha_{2}=\alpha_{3}+\alpha_{4}=:\beta\not\in\mathbb{Q}.
\]
Then $\beta$ is a cubic algebraic number and one can label the algebraic conjugates $\alpha_1,\alpha_2,\dotsc,\alpha_6$ of $\alpha$ in such a way that these satisfy the relations
\begin{equation}\label{eq:rel3}
\begin{cases}
\beta_{1} = \alpha_{1}+\alpha_{2} = \alpha_{3}+\alpha_{4},\\
\beta_{2} = \alpha_{2}+\alpha_{5} = \alpha_{3}+\alpha_{6},\\
\beta_{3} = \alpha_{1}+\alpha_{6} = \alpha_{4}+\alpha_{5},\\
\end{cases}
\end{equation}
% \begin{equation}\label{eq:rel1}
% \begin{cases}
% -2\beta_{1} = \alpha_{5}+\alpha_{6},\\
% -2\beta_{2} = \alpha_{1}+\alpha_{4},\\
% -2\beta_{3} = \alpha_{2}+\alpha_{3},\\
% \end{cases}
% \end{equation}
where $\beta_1=\beta,\beta_2,\beta_3$ are the algebraic conjugates of $\beta$. 
 Let $G$ be the Galois group of the normal closure of $\mathbb{Q}(\alpha)$ over $\mathbb{Q}$. Assume that $G$ is a subgroup of $S_6$, acting on the indices of the conjugates $\alpha_1,\alpha_2,\dotsc,\alpha_6$ of $\alpha$. Then, given the relations \eqref{eq:rel3}, there are exactly three possible cases:
\begin{enumerate}
\item $G=\langle \tau, \pi\ |\ \tau^{2}=\pi^{6}=id,\ \tau\pi\tau = \pi^{5}\rangle\cong D_{6}$; %(a dihedral group of order 12);
\item $G=\langle \pi\ |\ \pi^{6}=id\rangle\cong C_{6}$;
\item $G = \{id, \sigma, \sigma^{2}, \tau, \tau\sigma, \tau\sigma^{2}\}\cong S_3$.
\end{enumerate}
\end{proposition}

\section{Proofs}\label{proofs}

\begin{proof}[Proof of Theorem~\ref{111}]
(i) Suppose that some four distinct algebraic conjugates of an algebraic number $\alpha$ 
of degree $d\in\{4,5,6,7\}$ satisfy the relation $\alpha_{1}+\alpha_{2}+\alpha_{3}+\alpha_{4}=0$. 
The case $d=4$ is trivial in view of $tr(\alpha)=\alpha_{1}+\alpha_{2}+\alpha_{3}+\alpha_{4}$.  By Lemma $\ref{intro4}$, $d$ cannot be 5 or 7. Let $d=6$.  Lemma $\ref{intro3}$ implies that 
$tr(\al)=0$. Then $\alpha_5+\alpha_6 = tr(\alpha)-(\alpha_{1}+\alpha_{2}+\alpha_{3}+\alpha_{4})=0$. 
Hence, $\alpha_6=-\alpha_5$. Let $p(x)$ be the minimal polynomial of $\alpha$ over $\mathbb{Q}$. We have that $p(\alpha_6)=p(-\alpha_5)=0$. Hence, $\alpha_5$ is a root of $p(-x)$. Thus $p(x)$ divides the polynomial $p(-x)$. Since both polynomials $p(x)$ and $p(-x)$ are of the same degree and and their constant terms coincide, we have that $p(-x)=p(x)$. So $p(x)$ is of the form 
\[
x^{6}+ax^{4}+bx^{2}+ c \in\mathbb{Q}[x].
\]
The converse is clear, since the roots $\alpha_1,\alpha_2,\dotsc, \alpha_6$ of such polynomial  satisfy
\[
\alpha_{1}=-\alpha_{2},\quad \alpha_{3}=-\alpha_{4},\quad \alpha_{5}=-\alpha_{6}.
\]
Thus, $\alpha_{1}+\alpha_{2}+\alpha_{3}+\alpha_{4}=0$.  

(ii) Suppose that some four distinct algebraic conjugates of an algebraic number $\alpha$ 
of degree $d\in\{4,5,6,7\}$ satisfy the relation $\alpha_{1}+\alpha_{2}+\alpha_{3}=\alpha_{4}$. 
If $d=4$, then, by Lemma~\ref{intro3}, $tr(\alpha)=\alpha_{1}+\alpha_{2}+\alpha_{3}+\alpha_{4} = 0$ and we obtain $\alpha_4 + \alpha_4 = (\alpha_{1}+\alpha_{2}+\alpha_{3}) + \alpha_4 =0$. A contradiction. By Lemma $\ref{intro4}$, $d$ cannot be 5 or 7. Hence, $d=6$. By Lemma~\ref{intro3}, $tr(\alpha)=\alpha_{1}+\dotsb+\alpha_{6} = 0$. Since $\alpha_{1}+\alpha_{2}+\alpha_{3}=\alpha_{4}$, we have that
\[
0=\alpha_{1}+\dotsb+\alpha_{6} = 2\alpha_4+\alpha_5+\alpha_6,
\]
which is impossible in view of Lemma~\ref{lemma4}.
\end{proof}

\begin{proof}[Proof of Proposition~\ref{333}] Let $\alpha$ be an algebraic number of degree $d=6$  such that  $tr(\alpha)=0$ and some four distinct conjugates of $\alpha$ satisfy the relation
\[
\alpha_{1}+\alpha_{2}=\alpha_{3}+\alpha_{4}=:\beta.
\]
Let 

\[
\mathcal{S}:=\{\alpha_{1}, \alpha_{2}, \alpha_{3}, \alpha_{4}, \alpha_{5}, \alpha_{6}\}
\] 
be the full set of algebraic conjugates of $\alpha$. Then, in view of $tr(\alpha)=0$, we have
\[
\alpha_{1}+\alpha_{2}=\alpha_{3}+\alpha_{4}=\beta,\quad \alpha_{5}+\alpha_{6} = -2\beta.
\] 
%We can assume that because each polynomial with rational coefficients
%\[
%x^{n} + a_{n-1}x^{n-1} +\dots +a_{1}x + a_{0}
%\]
%can be converted to its depressed form with the appropriate substitution ($x = y - \frac{a_{n-1}}{na_{n}}$) and the roots of the converted polynomial will satisfy our linear relation.\\

%The only nontrivial case for $d=6$ is $\beta$ -- algebraic number of degree $3$ over $\mathbb{Q}$ with trace zero, i.e $\beta_{1}+\beta_{2}+\beta_{3}=0$, where $\beta_{1}:=\beta, \beta_{2}, \beta_{3}$ is the full set of algebraic conjugates of $\beta$. 

Let $G$ be the Galois group of the normal closure of $\mathbb{Q}(\alpha_{1})$ over $\mathbb{Q}$. The group $G$ is determined (in a unique way) by its action on $\mathcal{S}$: it corresponds to some transitive subgroup of the full symmetric group $S_{6}$. First consider the trivial case:
\[
\alpha_{1}+\alpha_{2}=\alpha_{3}+\alpha_{4}=a,\quad \alpha_{5}+\alpha_{6} = -2a,
\]
where $a\in\mathbb{Q}$. Select an automorphism $\phi\in G$ that maps $\alpha_{1}$ to $\alpha_{5}$. Setting $\phi(\alpha_{2}) = \alpha_{k}$, we obtain $\alpha_{5} + \alpha_{k} = a$. We claim that $k=6$. Indeed, if $1\leq k\leq 2$, then $\alpha_{1}+\alpha_{2}=a$ together with $\alpha_{5} + \alpha_{k} = a$ imply $\alpha_{5}=\alpha_{1}$ or $\alpha_{5}=\alpha_{2}$, which is impossible. Similarly, if $3\leq k\leq 4$, then $\alpha_{3}+\alpha_{4}=a$ together with $\alpha_{5} + \alpha_{k} = a$ imply $\alpha_{5}=\alpha_{3}$ or $\alpha_{5}=\alpha_{4}$, and we get another contradiction. Clearly, $k\neq 5$, so the only option is $\alpha_{5} + \alpha_{6} = a$. But we already know that $\alpha_{5}+\alpha_{6} = -2a$. Thus $a=-2a$, meaning that $a=0$. %Therefore $\alpha_{1}=-\alpha_{2}, \alpha_{3}=-\alpha_{4},$ and $\alpha_{5} = -\alpha_{6}$. Hence the minimal polynomial of $\alpha$ over $\mathbb{Q}$ must be of the form 
%\[
%x^{6}+ax^{4}+bx^{2}+ c \in\mathbb{Q}[x].
%\]
%The converse is clear, since the conjugates of such polynomial must satisfy $\alpha_{1}=-\alpha_{2}, \alpha_{3}=-\alpha_{4}$, and $\alpha_{5}=-\alpha_{6}$. Thus, $\alpha_{1}+\alpha_{2}=\alpha_{3}+\alpha_{4}=0$.

Now assume that 
\[
\alpha_{1}+\alpha_{2}=\alpha_{3}+\alpha_{4}=\beta,\quad \alpha_{5}+\alpha_{6} = -2\beta,
\]
where $\beta\notin\mathbb{Q}$. We will prove that $\beta_{1} := \beta$ is a cubic algebraic number. 

Let us write all possible distinct expressions of $\beta_{1}$ in terms of $\alpha_{i}+\alpha_{j}$ (sum of two distinct $\alpha$ conjugates). Assume that there are exactly $l$ distinct expressions (two expressions $\alpha_i+\alpha_j$ and $\alpha_u+\alpha_v$ are distinct if $\{i,j\}\neq \{u,v\}$):
\begin{equation*}%\label{eq5}
\beta_{1} = \alpha_{1}+\alpha_{2} = \alpha_{3}+\alpha_{4} =\alpha_u+\alpha_v= \dots.
\end{equation*}
Notice that $l\geq 2$, since the equality $\alpha_{1}+\alpha_{2} = \alpha_{3}+\alpha_{4}$ provides at least two distinct expressions of $\beta_{1}$. We will show that $l=2$. Indeed, assume that $l\geq 3$. Then we have at least three distinct expressions of $\beta_1$ as a sum of two distinct conjugates of $\alpha$:
\begin{equation}\label{eq5a}
\beta_{1} = \alpha_{1}+\alpha_{2} = \alpha_{3}+\alpha_{4} =\alpha_u+\alpha_v.
\end{equation}
Then $\{u,v\} =\{5,6\}$.  Indeed, if $u\in\{1,2,3,4\}$, then \eqref{eq5a} implies that $\alpha_v$ coincides with one of the conjugates $\alpha_1$, $\alpha_2$, $\alpha_3$, $\alpha_4$, which is impossible, since all three expressions in \eqref{eq5a} are distinct. Similarly, $v\in\{1,2,3,4\}$ also leads to a contradiction. 
Hence, $\{u,v\}=\{5,6\}$. So \eqref{eq5a} becomes
\[
\beta_{1} = \alpha_{1}+\alpha_{2} = \alpha_{3}+\alpha_{4} =\alpha_5+\alpha_6.
\]
Adding all these expressions of $\beta_1$, we obtain
\[
3\beta_{1} = \alpha_{1}+\alpha_{2} + \alpha_{3}+\alpha_{4} + \alpha_{5}+\alpha_{6} = 0 ,
\]
which is impossible in view of $\beta_{1}\notin\mathbb{Q}$.

%Indeed, firstly assume that $l\geq 4$. In such a case, there exist four distinct sums $\alpha_{i}+\alpha_{j}$ that contain eight conjugates of $\alpha$ over $\mathbb{Q}$. Since there are only six distinct conjugates of $\alpha$ over $\mathbb{Q}$, then some, say, $\alpha_{i}$ must appear at least two times in these expressions:
%\[
%\beta_{1} = \alpha_{i}+\alpha_{j} = \alpha_{i}+\alpha_{k} \implies \alpha_{j} = \alpha_{k},
%\]
%a contradiction. Hence, $ l\leq 3$. Now, assume that $l=3$. In this case, we can write (by re-numbering the conjugates of $\alpha$ over $\mathbb{Q}$, if necessary)
%\[
%\beta_{1} = \alpha_{1}+\alpha_{2} = \alpha_{3}+\alpha_{4} = \alpha_{5}+\alpha_{6}
%\]
%By adding these equalities, we get
%\[
%3\beta_{1} = \alpha_{1}+\alpha_{2} + \alpha_{3}+\alpha_{4} + \alpha_{5}+\alpha_{6} = 0 \implies \beta_{1}=0,
%%\Tr(\alpha)\in\mathbb{Q},
%%0 \implies \beta_{1}=0,
%%0?,
%\]
% sum = 0 ?
%again, a contradiction because $\beta_{1}\notin\mathbb{Q}$. 
Now we have that $l=2$ and  
\begin{equation}\label{eq373}
\beta_{1} = \alpha_{1}+\alpha_{2} = \alpha_{3}+\alpha_{4}.
\end{equation}
Next, we will obtain an upper bound for $\deg(\beta_{1})$. Note that by acting on \eqref{eq373} with an appropriate automorphism from $G$ we can obtain expression of the form \eqref{eq373} for every algebraic conjugate of $\beta_1$:
\begin{equation}\label{eq377}
\begin{split}
\beta_{1} &= \alpha_{1}\phantom{a}+\alpha_{2}\;\; = \alpha_{3}\phantom{a}+\alpha_{4},\\
\beta_{2} &= \alpha_{i_{21}}+\alpha_{i_{22}} = \alpha_{i_{23}}+\alpha_{i_{24}},\\
               &= \dotsb\;\;\;\dotsb\\
\beta_t &= \alpha_{i_{t1}}+\alpha_{i_{t2}} = \alpha_{i_{t3}}+\alpha_{i_{t4}}.
\end{split}
\end{equation}
Here $t$ is the degree of $\beta_1$  and $\beta_1,\beta_2,\dotsc, \beta_t$ are the algebraic conjugates of $\beta_1$. 
We have precisely $2\cdot t$ distinct expressions of the form $\alpha_i+\alpha_j$ in \eqref{eq377}, since there are exactly $t$ algebraic conjugates of $\beta_1$ and every such conjugate has exactly two expressions.  On the other hand, since $\deg(\alpha)=6$, we have at most $\binom{6}{2} = 15$ possible pairs of indices for distinct expressions $\alpha_{i}+\alpha_{j}$. Hence,
%Since there are only six distinct conjugates of $\alpha$ over $\mathbb{Q}$, we have at most $\binom{6}{2} = 15$ possible pairs of indices for distinct expressions $\alpha_{i}+\alpha_{j}$ 
%(two expressions $\alpha_i+\alpha_j$ and $\alpha_u+\alpha_v$ are distinct if $\{i,j\}\neq \{u,v\}$). Therefore
$ 2t\leq 15$ and $t=\deg(\beta_1)\leq 7$. 
%Note that $t=\deg(\beta_1)\neq 7$. Indeed, we have that 
%$\beta_1=\alpha_1+\alpha_2\in\mathbb{Q}(\alpha_1,\alpha_2)$. So $\deg(\beta_1)$ divides the degree of $\mathbb{Q}(\alpha_1,\alpha_2)$ over $\mathbb{Q}$. Moreover, 
%\[
%[\mathbb{Q}(\alpha_1,\alpha_2):\mathbb{Q}] = [\mathbb{Q}(\alpha_1,\alpha_2):\mathbb{Q}(\alpha_1)]\cdot [\mathbb{Q}(\alpha_1):\mathbb{Q}] = [\mathbb{Q}(\alpha_1,\alpha_2):\mathbb{Q}(\alpha_1)]\cdot 6
%\]
%and $[\mathbb{Q}(\alpha_1,\alpha_2):\mathbb{Q}(\alpha_1)]\leq 5$. Hence, the degree 
%$[\mathbb{Q}(\alpha_1,\alpha_2):\mathbb{Q}]$ is not divisible by 7, and therefore $\deg(\beta_1)\leq 6$.

Next, we will show that, in fact, $\deg(\beta_{1})$ is divisible by $3$. Indeed, in \eqref{eq377} there are 
$2t$ distinct expressions of conjugates of $\beta_{1}$ as sums $\alpha_{i}+\alpha_{j}$. Each such sum contains two conjugates of $\alpha$. Hence there are exactly $2\cdot 2t$ appearances of $\alpha$ conjugates in \eqref{eq377}. On the other hand, since $G$ is transitive on the set of algebraic conjugates of $\alpha$, each $\alpha_{i}$ must appear the same number of times in 
\eqref{eq377}. Suppose that every $\alpha_{i}$ appears exactly $k$ times in  \eqref{eq377}. So 
we have exactly $k\cdot \deg(\alpha)=6k$ appearances of $\alpha$ conjugates in \eqref{eq377}. Hence, $4t=6k$, and therefore $t$ is divisible by 3. Recall that $t=\deg(\beta_1)\leq 7$. 
So $\deg(\beta_1)=3$ or 6. 

%Indeed, by applying all the automorphisms of $G$ to $(\ref{eq5})$, we obtain a system of linear equations (denote it by $(\star)$) with $\deg(\beta_{1})\cdot l$ expressions of conjugates of $\beta_{1}$ over $\mathbb{Q}$ through distinct sums $\alpha_{i}+\alpha_{j}$. Since there are two members in a sum of this form ($\alpha_{i}$ and $\alpha_{j}$), we obtain $2\cdot\deg(\beta_{1})\cdot l$ appearances of $\alpha$ conjugates over $\mathbb{Q}$ %in this system of linear equations
%in $(\star)$. On the other hand, since $Gal(\mathbb{Q}(\alpha)/\mathbb{Q})$ is transitive, each $\alpha_{i}$ must appear the same number of times in %our system of linear equations
%$(\star)$. Say, each $\alpha_{i}$ appears exactly $k$ times in $(\star)$. Thus we have $k\cdot\deg(\alpha)$ appearances of $\alpha$ conjugates over $\mathbb{Q}$ in $(\star)$ %our system of linear equations. 
%Hence, we have an equality
%\[
%2\deg(\beta_{1})\cdot l = k\cdot\deg(\alpha).
%\]
%In our case, $l=2$ and $\deg(\alpha) = 6$, thus
%\[
%4\deg(\beta_{1}) = 6k\implies 2\deg(\beta_{1}) = 3k.
%\]
%Hence, $\deg(\beta_{1})$ is divisible by $3$. Since $\deg(\beta_{1})\leq 6$, we must have that $\deg(\beta_{1})=3$ or $\deg(\beta_{1})=6$. 

Finally, we will show that $\deg(\beta_{1})\neq 6$. Indeed, assume that $\deg(\beta_{1})=6$. Since each conjugate of $\beta_{1}$  has exactly $2$ distinct expressions of the form $\alpha_{i}+\alpha_{j}$, we obtain $6\cdot 2=12$ distinct expressions. Recall that there are at most $\binom{6}{2} = 15$ possible pairs of indices for distinct expressions $\alpha_{i}+\alpha_{j}$ and also
\[
-2\beta_{1} = -(\alpha_{1}+\alpha_{2}) - (\alpha_{3}+\alpha_{4}) = \alpha_{5}+\alpha_{6}.
\]
By applying all automorphisms from $G$ to $-2\beta_{1} = \alpha_{5}+\alpha_{6}$, we get at least $\deg(\beta_{1})=6$ expressions of the form $\alpha_{i}+\alpha_{j}$ for algebraic conjugates of $-2\beta_{1}$. These expressions must be distinct from each other and from $12$ expressions that we already have. But in such a case, there are
\[
12+6 = 18
\]
distinct expressions $\alpha_{i}+\alpha_{j}$. A contradiction, since there are at most $15$ distinct such expressions. Hence, $\deg(\beta_{1})\neq 6$, and the only possibility is $\deg(\beta_{1})=3$. 
\end{proof}

\begin{proof}[Proof of Proposition~\ref{prop562}] Let $\alpha$ be an algebraic number of degree $d=6$  and $tr(\alpha)=0$. Assume that some four distinct algebraic conjugates of $\alpha$ satisfy the relation
\begin{equation}\label{eq:555lr}
\alpha_{1}+\alpha_{2}=\alpha_{3}+\alpha_{4}=:\beta\not\in\mathbb{Q}.
\end{equation} 
Then, by Proposition~\ref{333}, $\beta$ is a cubic algebraic number. Let 
$\beta_1=\beta,\beta_2,\beta_3$ be the algebraic conjugates of $\beta$. 
Let $G$ be the Galois group of the normal closure of $\mathbb{Q}(\alpha)$ over $\mathbb{Q}$. Assume that $G$ is a subgroup of $S_6$, acting on the indices of the conjugates $\alpha_1,\alpha_2,\dotsc,\alpha_6$ of $\alpha$. Take two automorphisms of $G$ such that one maps $\beta_1$ to $\beta_2$ and another maps $\beta_1$ to $\beta_3$. Acting with these automorphisms on \eqref{eq:555lr}, we obtain
\begin{equation}\label{eq20}
\begin{cases}
\beta_{1} = \alpha_{1}+\alpha_{2} = \alpha_{3}+\alpha_{4},\\
\beta_{2} = \alpha_{i_{21}}+\alpha_{i_{22}} = \alpha_{i_{23}}+\alpha_{i_{24}},\\
\beta_{3} = \alpha_{i_{31}}+\alpha_{i_{32}} = \alpha_{i_{33}}+\alpha_{i_{34}},
\end{cases}
\end{equation}
where each $\alpha_{i_{kl}}$ is an algebraic conjugate of $\alpha$. In view of  $tr(\alpha)=0$, we also obtain  corresponding relations
\begin{equation}\label{eq21}
\begin{cases}
-2\beta_{1} = \alpha_{5}+\alpha_{6},\\
-2\beta_{2} = \alpha_{i_{25}}+\alpha_{i_{26}},\\
-2\beta_{3} = \alpha_{i_{35}}+\alpha_{i_{36}}.\\
\end{cases}
\end{equation}\\
%For both systems we have $n_{i}, m_{j}\in\{1, 2, 3, 4, 5, 6\}$. 
Note that for every $k=2,3$ the numbers $i_{k1},i_{k2}, i_{k3}, i_{k4}, i_{k5}, i_{k6}$ are distinct. 
We will specify the indices in $(\ref{eq20})$ and $(\ref{eq21})$ by relabeling the conjugates $\alpha_1,\alpha_2,\dotsc,\alpha_6$, if necessary. First, we will prove that $-2\beta_{k}$ has a unique expression in terms of $\alpha_i+\alpha_j$ (recall that two expressions $\alpha_i+\alpha_j$ and $\alpha_u+\alpha_v$ are distinct if $\{i,j\}\neq \{u,v\}$). Indeed, say, $-2\beta_{1}$ has two distinct expressions:
\begin{equation}\label{eq97}
-2\beta_{1} = \alpha_{5}+\alpha_{6} = \alpha_{u}+\alpha_{v}.
\end{equation}
If $u=5$ (or $u=6$), then by $(\ref{eq97})$, $v=6$ (or $v=5$, respectively). In this scenario, the expressions  $\alpha_{u}+\alpha_{v}$ and $\alpha_{5}+\alpha_{6}$ become identical. This implies that $u\notin\{5,6\}$. A similar argument shows that $v\notin\{5,6\}$. Consequently, $u, v\in\{1, 2, 3, 4\}$. Without loss of generality, let $u=1$. We then examine the following cases for $(u,v)$: 
\begin{itemize}
\item[Case 1:] If $(u,v)=(1,1)$, then $\alpha_{5}+\alpha_{6} = 2\alpha_{1}$, which contradicts Lemma~$\ref{intro7}$.
\item[Case 2:] If $(u,v)=(1,2)$, then equations \eqref{eq20} and \eqref{eq97} yield $-2\beta_{1}=\beta_{1}$ implying $\beta_{1}=0$. This contradicts the condition that $\beta_{1}\notin\mathbb{Q}$.
\item[Case 3:] If $(u,v)=(1,3)$, then we have $\beta_{1} = \alpha_{1}+\alpha_{2}$ and $-2\beta_{1} = \alpha_{1} + \alpha_{3}$. Substituting the first into the second gives $3\alpha_1+2\alpha_2+\alpha_3=0$, which contradicts Lemma $\ref{lemma4}$.
\item[Case 4:] Similarly, $(u,v)\neq(1,4)$.
\end{itemize}
These cases imply that  $u\notin\{1, 2, 3, 4\}$. Therefore, $-2\beta_{1}$ has a unique expression in terms of $\alpha_i+\alpha_j$. Since $\beta_1,\beta_2,\beta_3$ are algebraic conjugates of $\beta_1$, it follows that 
 every $-2\beta_{k}$ (for $k=1,2,3$) also has a unique expression in terms of $\alpha_i+\alpha_j$.  

%Since the Galois group $G$ is transitive on the set of algebraic conjugates of $\alpha$, each $\alpha_{i}$ must appear the same number of times in \eqref{eq21}. In our case, each $\alpha_{i}$ must appear exactly once in \eqref{eq21}.
Now we will prove that all the $\alpha$'s appear exactly once in $(\ref{eq21})$. Indeed, assume that some conjugate of $\alpha$ appears at least twice in $(\ref{eq21})$. 
Without loss of generality, we may assume that $\alpha_{i_{25}}=\alpha_5$, i.e., $-2\beta_{2} = \alpha_{5}+\alpha_{i_{26}}$. 
Recall that for every $k=2,3$ the numbers $i_{k1},i_{k2}, i_{k3}, i_{k4}, i_{k5}, i_{k6}$ are distinct. Therefore, $\alpha_5$ does not appear in the expressions of $\beta_1$ and $\beta_2$ in \eqref{eq20}. 
Moreover, from the proof of Proposition~\ref{333} we know that every algebraic conjugate of $\alpha$ appears exactly twice in \eqref{eq20}. Consequently, $\alpha_5$ appears twice in the expression of $\beta_3$ in \eqref{eq20}. This means that the numbers $i_{31},i_{32}, i_{33}, i_{34}$ are not distinct, a contradiction. 
Hence, all the $\alpha$'s appear exactly once in $(\ref{eq21})$.

% Without loss of generality, we may assume that $\alpha_{i_{31}}=\alpha_{i_{33}} =\alpha_5$:
% \[
% \beta_{3} = \alpha_{5} + \alpha_{i_{32}} = \alpha_{5} + \alpha_{i_{34}}.
% \]

% If $i=6$, then $\beta_{1}=\beta_{2}$, which is impossible. If $i=5$, then $-\beta_{2} = \alpha_{5}$, and we obtain a contradiction, since $3=deg(-\beta_{2})\neq deg(\alpha_{5})=6$. Hence, $i\in\{1, 2, 3, 4\}$ and we have that
% \[
% -2\beta_{2} = -(\alpha_{i_{21}}+\alpha_{i_{22}}) - (\alpha_{i_{23}}+\alpha_{i_{24}}) + tr(\alpha) = \alpha_{5}+\alpha_{i}
% \]
% implies that $i_{21}, i_{22}, i_{23}, i_{24}\neq 5$. Since the Galois group $G$ is transitive on the set of algebraic conjugates of $\alpha$, each $\alpha_{i}$ must appear the same number of times in \eqref{eq20}. In our case, each $\alpha_{i}$ must appear exactly two times in \eqref{eq20}. Therefore, without loss of generality, we must have
% \[
% \beta_{3} = \alpha_{5} + \alpha_{5} = \alpha_{i_{33}} + \alpha_{i_{34}}\quad\text{or}\quad\beta_{3} = \alpha_{5} + \alpha_{i_{32}} = \alpha_{5} + \alpha_{i_{34}}.
% \]
% In the first case, we immediately obtain a contradiction, since $3=deg(\beta_{3})\neq deg(2\alpha_{5})=6$. The second case, implies that $\alpha_{i_{32}} = \alpha_{i_{34}}$, which is impossible because the expressions of $\beta_{3}$ must be distinct. Hence, we derive that all the $\alpha$'s appear exactly once in $(\ref{eq21})$.

We have that $\{\alpha_{i_{25}}, \alpha_{i_{26}}, \alpha_{i_{35}}, \alpha_{i_{36}}\}=\{\alpha_{1}, \alpha_{2}, \alpha_{3}, \alpha_{4}\}$. Without loss of generality, we can assume that $\alpha_{i_{25}} = \alpha_{1}$. 
If $\alpha_{i_{26}} = \alpha_{2}$, then $-2\beta_2=\beta_1$, which contradicts Lemma~\ref{lemma4}. Hence, $\alpha_{i_{26}}\in\{\alpha_{3}, \alpha_{4}\}$. 
Note that $\alpha_{3}$ and $\alpha_{4}$ appear symmetrically in the first equation of $(\ref{eq20})$. 
Without loss of generality, by relabeling $\alpha_{3}$ and $\alpha_{4}$, if necessary, we can assume that $\alpha_{i_{26}} = \alpha_{4}$.  
From this, we immediately derive that $\{\alpha_{i_{21}}, \alpha_{i_{22}}, \alpha_{i_{23}}, \alpha_{i_{24}}\}=\{\alpha_{2}, \alpha_{3}, \alpha_{5}, \alpha_{6}\}$. 
Note that $\beta_{2}\neq\alpha_{5}+\alpha_{6}$. 
Indeed, if $\beta_{2}=\alpha_{5}+\alpha_{6}$, then $\beta_{2}=-2\beta_{1}$, which contradicts Lemma~\ref{lemma4}.  
Thus $\alpha_{5}$ and $ \alpha_{6}$ appear in distinct expressions of $\beta_{2}$ in \eqref{eq20}, as well as $\alpha_{2}$ and $\alpha_{3}$. 
Without loss of generality, we can assume that $\alpha_{i_{21}} = \alpha_{2}$ and $\alpha_{i_{23}} = \alpha_{3}$. 
Since $\alpha_{5}$ and $\alpha_{6}$ appear symmetrically in the first equation of $(\ref{eq21})$, by relabeling $\alpha_{5}$ and $\alpha_{6}$, if necessary, we can assume that $\alpha_{i_{22}} = \alpha_{5}$ and $\alpha_{i_{24}} = \alpha_{6}$. So far, we have obtained
\begin{equation*}
\begin{cases}
\beta_{1} = \alpha_{1}+\alpha_{2} = \alpha_{3}+\alpha_{4},\\
\beta_{2} = \alpha_{2}+\alpha_{5} = \alpha_{3}+\alpha_{6},\\
\beta_{3} = \alpha_{i_{31}}+\alpha_{i_{32}} = \alpha_{i_{33}}+\alpha_{i_{34}},
\end{cases}
\begin{cases}
-2\beta_{1} = \alpha_{5}+\alpha_{6},\\
-2\beta_{2} = \alpha_{1}+\alpha_{4},\\
-2\beta_{3} = \alpha_{i_{35}}+\alpha_{i_{36}}.\\
\end{cases}
\end{equation*}
Since $\{\alpha_{i_{35}}, \alpha_{i_{36}}\}=\{\alpha_{2}, \alpha_{3}\}$, without loss of generality, we 
assume that $\alpha_{i_{35}} = \alpha_{2}$ and $\alpha_{i_{36}} = \alpha_{3}$. Then $\{\alpha_{i_{31}}, \alpha_{i_{32}}, \alpha_{i_{33}}, \alpha_{i_{34}}\}=\{\alpha_{1}, \alpha_{4}, \alpha_{5}, \alpha_{6}\}$. 
Note that $\beta_{3}\neq\alpha_{5}+\alpha_{6}$. Indeed, if $\beta_{3}=\alpha_{5}+\alpha_{6}$, then $\beta_{3}=-2\beta_{1}$, which contradicts Lemma~\ref{lemma4}. Therefore,  $\alpha_{5}$ and $ \alpha_{6}$ appear in distinct expressions of $\beta_{3}$ in \eqref{eq20}, as well as $\alpha_{1}$ and $\alpha_{4}$. 
Thus, without loss of generality, we can assume that $\alpha_{i_{31}} = \alpha_{1}$ and $\alpha_{i_{33}} = \alpha_{4}$. Now we have two possible cases: 
\[
\beta_{3} = \alpha_{1}+\alpha_{5} = \alpha_{4}+\alpha_{6}\;\; \text{or}\;\; \beta_{3} = \alpha_{1}+\alpha_{6} = \alpha_{4}+\alpha_{5}.
\]
The first case is impossible. Indeed, by adding $\beta_{1} = \alpha_{1}+\alpha_{2}, \beta_{2} = \alpha_{2}+\alpha_{5}$ and $\beta_{3} = \alpha_{1}+\alpha_{5}$, we obtain
\[
0 = \beta_{1} + \beta_{2} + \beta_{3} = 2(\alpha_{1}+ \alpha_{2}+\alpha_{5})=2(\beta_{1}+\alpha_{5}),
\]
and hence $\beta_{1}=-\alpha_{5}$, which is impossible, since $6=\deg(-\alpha_{5})\neq\deg(\beta_{1})=3$. 
%The case $ii)$, together with other equations that we already have, provides these systems of linear equations:
Finally, we can rewrite equations \eqref{eq20} and \eqref{eq21} as follows
\begin{equation}\label{eq24}
\begin{cases}
\beta_{1} = \alpha_{1}+\alpha_{2} = \alpha_{3}+\alpha_{4},\\
\beta_{2} = \alpha_{2}+\alpha_{5} = \alpha_{3}+\alpha_{6},\\
\beta_{3} = \alpha_{1}+\alpha_{6} = \alpha_{4}+\alpha_{5},
\end{cases}
\end{equation}
\begin{equation}\label{eq25}
\begin{cases}
-2\beta_{1} = \alpha_{5}+\alpha_{6},\\
-2\beta_{2} = \alpha_{1}+\alpha_{4},\\
-2\beta_{3} = \alpha_{2}+\alpha_{3}.\\
\end{cases}
\end{equation}\\
% The system $(\ref{eq24})$ can be written as a $6\times7$ matrix in such a way
% \[
% \begin{bmatrix} 
% 1 & 1 & 0 & 0 & 0 & 0 & \beta_{1}\\
% 0 & 0 & 1 & 1 & 0 & 0 & \beta_{1}\\
% 0 & 0 & 1 & 0 & 0 & 1 & \beta_{2}\\
% 0 & 1 & 0 & 0 & 1 & 0 & \beta_{2}\\
% 0 & 0 & 0 & 1 & 1 & 0 & \beta_{3}\\
% 1 & 0 & 0 & 0 & 0 & 1 & \beta_{3}\\
% \end{bmatrix}\\
% \]
% By using Gaussian elimination we obtain
% \[
% \begin{bmatrix} 
% 1 & 1 & 0 & 0 & 0 & 0 & \beta_{1}\\
% 0 & 1 & 0 & 0 & 1 & 0 & \beta_{2}\\
% 0 & 0 & 1 & 0 & 0 & 1 & \beta_{2}\\
% 0 & 0 & 0 & -1 & 0 & 1 & -\beta_{1}+\beta_{2}\\
% 0 & 0 & 0 & 0 & 1 & 1 & -\beta_{1}+\beta_{2}+\beta_{3}\\
% 0 & 0 & 0 & 0 & 0 & 0 & 0\\
% \end{bmatrix}\\
% \]
% Both of these matrices have rank 5. Therefore the system of equations $(\ref{eq24})$ is consistent.\\

Now, we will prove that the Galois group $G$ is of order $|G| = 6$ or $12$. 
% Indeed, it is clear that $|G|\neq 1$, because $G$ is a transitive subgroup of $S_6$. In order to show that $|G| = 6$ or $12$, it is sufficient to prove that the stabilizer subgroup of $\alpha_{1}$, i.e., $Stab(\alpha_{1})$ has order at most 2. 
Since $G$ is transitive, Orbit-Stabilizer Theorem implies
\[
|G| =  |Orb(\alpha_{1})|\cdot |Stab(\alpha_{1})| = 6\cdot |Stab(\alpha_{1})|,
\]
where $Stab(\alpha_1) = \{\psi\in G:\psi(\alpha_1)=\alpha_1\}$. Hence, if we can show that 
$|Stab(\alpha_{1})|\leq 2$, then it follows that $|G|=6$ or 12.
 Let $\psi\in Stab(\alpha_{1})$.
% We assume that there exists some element $\psi\in Stab(\alpha_{1})$. We also use the system of linear equations $(\ref{eq24})$. 

If $\psi$ stabilizes $\alpha_{2}$, then, in view of $-2\beta_{2} = \alpha_{1}+\alpha_{4}$ and 
$-2\beta_{3} = \alpha_{2}+\alpha_{3}$, $\psi$ stabilizes $\beta_2$, $\alpha_3$ and $\alpha_4$. 
Since $\beta_2=\alpha_2+\alpha_5$, it follows that $\psi$ stabilizes $\alpha_5$. Hence, $\psi$ stabilizes every conjugate of $\alpha$. Therefore, $\psi=id$.

% If $\psi$ stabilizes $\alpha_{2}$, then, in view of $\beta_1=\alpha_1+\alpha_2$, $\psi$ stabilizes $\beta_{1}$. Notice that $\psi$ also stabilizes $\beta_{2}$ and $\beta_{3}$. Indeed, since $\beta_{1}$ is already stable, $\alpha_{2}$ appears only in $\beta_{2}= \alpha_{2}+\alpha_{5}$ expression and $\alpha_{1}$ appears only in $\beta_{3}=\alpha_{1}+\alpha_{6}$ expression (besides the expression of $\beta_{1}$), we must have that $\psi$ stabilizes $\beta_{2}$ and $\beta_{3}$. This observation immediately implies that $\psi$ maps $\alpha_{5}=\beta_{2}-\alpha_{2}$ and $\alpha_{6}=\beta_{3}-\alpha_{1}$ to itself. Consequently, $\psi$ maps $\alpha_{3}=\beta_{2}-\alpha_{6}$ and $\alpha_{4}=\beta_{3}-\alpha_{5}$ to itself. Hence, we derive $\psi = id$, i.e., $|Stab(\alpha_{1})|=|\{id\}|=1$.

Suppose that $\psi(\alpha_2)\neq \alpha_2$. Since $\psi(\alpha_1)=\alpha_1$ and every conjugate of $\alpha$ appears exactly once in \eqref{eq25}, it follows that $\psi$ stabilizes $\beta_2$ and $\alpha_4$. 
Moreover, from \eqref{eq24} we obtain that $\psi$ maps $\beta_1=\alpha_1+\alpha_2$ to $\beta_3=\alpha_1+\alpha_6$ and vice versa. 
Hence, $\psi(\alpha_2)=\alpha_6$ and $\psi(\alpha_6)=\alpha_2$. Furthermore, $\psi$ maps  
$-2\beta_1=\alpha_5+\alpha_6$ to $-2\beta_3=\alpha_2+\alpha_3$ and vice versa. 
So $\psi(\alpha_3)=\alpha_5$ and $\psi(\alpha_5)=\alpha_3$. Hence, $\psi=(2\,6)(3\,5)\in S_6$. 

We have proved that if every $\psi\in Stab(\alpha_1)$ stabilizes $\alpha_2$, then the stabilizer 
subgroup $Stab(\alpha_1)$ is trivial and $G$ has order 6. 
If there exists $\psi\in Stab(\alpha_1)$ which does not stabilize $\alpha_2$, then 
$Stab(\alpha_1)=\{id,(2\,6)(3\,5)\}$ and accordingly $G$ has order $6\cdot 2=12$. 
Hence, if $G$ has order 12, then necessarily $(2\,6)(3\,5)\in G$.

% Suppose that $\psi(\alpha_2)\neq \alpha_2$.  Then $\psi$ maps $\beta_{1}$ to $\beta_{3}$ because $\alpha_{1}$ appears only in $\beta_{3}=\alpha_{1}+\alpha_{6}$ expression (besides the expression of $\beta_{1}$). It means that $\psi$ maps $\alpha_{2}$ to $\alpha_{6}$. Now, notice that $\beta_{3}$ cannot be mapped to $\beta_{2}$. Indeed, if $\psi$ maps $\beta_{3}$ to $\beta_{2}$, then $\alpha_{1}+\alpha_{6}$ is mapped to one of the expressions $\alpha_{2}+\alpha_{5}$ or $\alpha_{3}+\alpha_{6}$. It means that $\alpha_{1}$ must be mapped, by $\psi$, to one of the elements in a set $\{\alpha_{2}, \alpha_{3}, \alpha_{5}, \alpha_{6}\}$. Contradiction, since $\psi$ stabilizes $\alpha_{1}$. %since there is no $\alpha_{1}$ in any expression of $\beta_{2}$. 
% Therefore, $\psi$ must map $\beta_{3}$ to $\beta_{1}=\alpha_{1}+\alpha_{2}$. This implies that $\psi$ maps $\alpha_{6}$ to $\alpha_{2}$. So we see that $\psi$ stabilizes $\beta_{2}$. From these results we derive that $\psi$ maps $\alpha_{5}$ to $\alpha_{3}$ and vice versa, i.e., $\psi$ also maps $\alpha_{3}$ to $\alpha_{5}$. Since $\beta_{1}$ is mapped to $\beta_{3}$ and vice versa, we obtain that $\alpha_{4}$ is fixed by $\psi$. Hence, we derived that $\psi = (2\ 6)(3\ 5)$ is the only nontrivial element in $Stab(\alpha_{1})$. It is easy to check that it preserves the relations in our system of linear equations $(\ref{eq24})$. Thus we showed that in this case $|Stab(\alpha_{1})|=|\{id, (2\ 6)(3\ 5)\}|= 2$.\\ 

Next we will prove that the group $G$ contains the permutation $\sigma := (1\ 3\ 5)$ $(2\ 6\ 4)\in S_6$. Indeed, since $\beta_{1}$ is a cubic algebraic number, the Galois group $G$ contains an element, denote it by $\varphi$,
that permutes the conjugates of $\beta_{1}$.  
Without loss of generality, we can assume that
\[
\varphi(\beta_{1})=\beta_{2},\quad \varphi(\beta_{2})=\beta_{3},\quad \varphi(\beta_{3})=\beta_{1}
\]
(by exchanging $\varphi$ with $\varphi^2$, if necessary).
% Similarly,
% \[
% \varphi(-2\beta_{1})=-2\beta_{2},\quad \varphi(-2\beta_{2})=-2\beta_{3},\quad \varphi(-2\beta_{3})=-2\beta_{1}.
% \]

Relations in $(\ref{eq25})$ imply that $\varphi$ maps $\{\alpha_{1}, \alpha_{4}\}$ to $\{\alpha_{2}, \alpha_{3}\}$. Consider two possible cases: $\varphi(\alpha_{1})=\alpha_{2}$ and 
$\varphi(\alpha_{1})=\alpha_{3}$.

If $\varphi(\alpha_{1})=\alpha_{2}$, then $\varphi(\alpha_{4})=\alpha_{3}$. The expressions of $\beta_{1}$ and $\beta_{2}$ in $(\ref{eq24})$ imply that $\varphi$ maps $\{\alpha_{1}, \alpha_{2}\}$ to $\{\alpha_{2}, \alpha_{5}\}$. Since $\varphi(\alpha_{1})=\alpha_{2}$, we obtain $\varphi(\alpha_{2})=\alpha_{5}$. Similarly, we see that $\varphi$ maps $\{\alpha_{3}, \alpha_{4}\}$ to $\{\alpha_{3}, \alpha_{6}\}$. Since $\varphi(\alpha_{4})=\alpha_{3}$, we derive $\varphi(\alpha_{3})=\alpha_{6}$. The expressions of $\beta_{2}$ and $\beta_{3}$ in $(\ref{eq24})$ imply that $\varphi$ maps $\{\alpha_{2}, \alpha_{5}\}$ to $\{\alpha_{4}, \alpha_{5}\}$. Since $\varphi(\alpha_{2})=\alpha_{5}$, we obtain $\varphi(\alpha_{5})=\alpha_{4}$. We are left with only one option for $\varphi(\alpha_6)$, i.e., $\varphi(\alpha_6)=\alpha_1$.  
% Similarly, we have that $\varphi$ maps $\{\alpha_{3}, \alpha_{6}\}$ to $\{\alpha_{1}, \alpha_{6}\}$. Since $\varphi(\alpha_{3})=\alpha_{6}$, we get $\varphi(\alpha_{6})=\alpha_{1}$. 
Hence, $\varphi = (1\ 2\ 5\ 4\ 3\ 6)$. Note that $\varphi^4=(1\ 3\ 5)(2\ 6\ 4)=\sigma$. 
So that in this case ($\varphi(\alpha_1)=\alpha_2$)  the permutation $\sigma$ is contained in $G$.

If $\varphi(\alpha_{1})=\alpha_{3}$, then $\varphi(\alpha_{4})=\alpha_{2}$. The expressions of $\beta_{1}$ and $\beta_{2}$ in $(\ref{eq24})$ imply that $\varphi$ maps $\{\alpha_{1}, \alpha_{2}\}$ to $\{\alpha_{3}, \alpha_{6}\}$. Since $\varphi(\alpha_{1})=\alpha_{3}$, we have $\varphi(\alpha_{2})=\alpha_{6}$. Similarly we see that $\varphi$ maps $\{\alpha_{3}, \alpha_{4}\}$ to $\{\alpha_{2}, \alpha_{5}\}$. Since $\varphi(\alpha_{4})=\alpha_{2}$, we get $\varphi(\alpha_{3})=\alpha_{5}$. The expressions of $\beta_{2}$ and $\beta_{3}$ in $(\ref{eq24})$ imply that $\varphi$ maps $\{\alpha_{2}, \alpha_{5}\}$ to $\{\alpha_{1}, \alpha_{6}\}$. Since $\varphi(\alpha_{2})=\alpha_{6}$, we obtain $\varphi(\alpha_{5})=\alpha_{1}$. 
We are left with only one option for $\varphi(\alpha_6)$, i.e., $\varphi(\alpha_6)=\alpha_4$.
Hence, $\varphi = (1\ 3\ 5)(2\ 6\ 4)=\sigma$.

We have proved that the group $G$ contains the permutation $\sigma = (1\ 3\ 5)(2\ 6\ 4)$. Now we are in a position to find all possible groups $G$. 

A simple computation with SageMath \cite{sagemath} shows that there is a unique transitive subgroup of $S_6$ which has order 12 and contains permutations $(2\,6)(3\,5)$ and 
$\sigma=(1\ 3\ 5)(2\ 6\ 4)$. 
This subgroup is generated by the permutations $\tau = (1\,2)(3\,4)(5\,6)$ and 
$\pi = (1\, 2\, 5\, 4\, 3\, 6)$ and is isomorphic to the dihedral group $D_6$ of order 12. 

Similarly, there are exactly four transitive subgroups of $S_6$ which have order 6 and contain the permutation $\sigma=(1\ 3\ 5)(2\ 6\ 4)$. These are
\begin{itemize}
\item $G_1=\langle \sigma,\tau \rangle$ isomorphic to $S_3$, where $\tau = (1\,2)(3\,4)(5\,6)$;
\item $G_2=\langle \pi \rangle$ isomorphic to the cyclic group $C_6$, where $\pi=(1\, 2\, 5\, 4\, 3\, 6)$;
\item $G_3=\langle (1\,6\,5\,2\,3\,4) \rangle$ isomorphic to the cyclic group $C_6$;
\item $G_4=\langle (1\,4\,5\,6\,3\,2) \rangle$ isomorphic to the cyclic group $C_6$.
\end{itemize}
Note that the groups $G_3$ and $G_4$ do not preserve the relations in \eqref{eq25}. Indeed, the generator of $G_3$ maps $\alpha_1+\alpha_4$ to $\alpha_6+\alpha_1$ while the 
generator of $G_4$ maps $\alpha_1+\alpha_4$ to $\alpha_4+\alpha_5$. 
Hence, $G\neq G_3$ and $G\neq G_4$.

We have proved that there are three options for the group $G$:
\begin{enumerate}
\item $G=\langle \tau, \pi\rangle\cong D_{6}$; %(a dihedral group of order 12);
\item $G=\langle \sigma,\tau\ \rangle\cong S_3$;
\item $G = \langle \pi\ \rangle\cong C_6$.
\end{enumerate}
% Here $\pi=(1\, 2\, 5\, 4\, 3\, 6)$, $\sigma=(1\ 3\ 5)(2\ 6\ 4)$ and $\tau = (1\,2)(3\,4)(5\,6)$.
This complete the proof of Proposition~\ref{prop562}.

\end{proof}

\begin{proof}[Proof of Theorem~\ref{propqc}]
\textit{Necessity.} Suppose that $\alpha$ is an algebraic number of degree 6  whose four distinct algebraic conjugates satisfy the relation $\alpha_{1}+\alpha_{2}=\alpha_{3}+\alpha_{4}=:\beta\notin\mathbb{Q}$. 
Note that for any $r\in\mathbb{Q}$ the number $\alpha-r$ will also have this property. 
Moreover, $\alpha$ equals the sum of a quadratic and a cubic algebraic number if and only if $\alpha-r$ has the same property. 
Hence, by taking $r=tr(\alpha)/6$, we can assume that $\alpha$ has zero trace $tr(\alpha)=0$. Let $G$ be the Galois group of the normal closure of $\mathbb{Q}(\alpha)$ over $\mathbb{Q}$. Assume that $G$ is a subgroup of $S_6$, acting on the indices of the conjugates $\alpha_1,\alpha_2,\dotsc,\alpha_6$ of $\alpha$. Then, by Proposition~\ref{prop562}, we have the following:
\begin{itemize}
\item[$(i)$] $\beta$ is a cubic algebraic number.
\item[$(ii)$] One can label the algebraic conjugates $\alpha_1,\alpha_2,\dotsc,\alpha_6$ of $\alpha$ in such a way that these satisfy the relations
\begin{equation}\label{eq:rel4}
\begin{cases}
\beta_{1} = \alpha_{1}+\alpha_{2} = \alpha_{3}+\alpha_{4},\\
\beta_{2} = \alpha_{2}+\alpha_{5} = \alpha_{3}+\alpha_{6},\\
\beta_{3} = \alpha_{1}+\alpha_{6} = \alpha_{4}+\alpha_{5},
\end{cases}
\end{equation}
where $\beta_1=\beta,\beta_2,\beta_3$ are the algebraic conjugates of $\beta$.
\item[$(iii)$] Given the relations \eqref{eq:rel4}, there are exactly three options for the Galois group $G$:
\begin{enumerate}
\item $G=\langle \tau, \pi\ |\ \tau^{2}=\pi^{6}=id,\ \tau\pi\tau = \pi^{5}\rangle\cong D_{6}$; %(a dihedral group of order 12);
\item $G=\langle \pi\ |\ \pi^{6}=id\rangle\cong C_{6}$;
\item $G = \{id, \sigma, \sigma^{2}, \tau, \tau\sigma, \tau\sigma^{2}\}\cong S_3$.
\end{enumerate}
\end{itemize}
 
Here $\pi=(1\, 2\, 5\, 4\, 3\, 6)$, $\sigma=(1\ 3\ 5)(2\ 6\ 4)$ and $\tau = (1\,2)(3\,4)(5\,6)$. 

Consider the number $\alpha_1-\alpha_4$. The expression of $\beta_1$ in \eqref{eq:rel4} implies that $\alpha_2-\alpha_3=-(\alpha_1-\alpha_4)$. Moreover, 
\begin{align*}
\tau(\alpha_1-\alpha_4) &= \alpha_2-\alpha_3=-(\alpha_1-\alpha_4),\\
\pi(\alpha_1-\alpha_4) &= \alpha_2-\alpha_3=-(\alpha_1-\alpha_4),\\
\sigma(\alpha_1-\alpha_4) &= \alpha_3-\alpha_2=\alpha_1-\alpha_4.
\end{align*}
Hence, $\alpha_1-\alpha_4$ is a quadratic algebraic number. 
Moreover, the relations $\beta_2=\alpha_{2}+\alpha_{5} = \alpha_{3}+\alpha_{6}$ 
together with $tr(\alpha)=0$ imply $-2\beta_2 = \alpha_1+\alpha_4$, which is equivalent to 
\[
\alpha_1 = \frac{\alpha_1-\alpha_4}{2} - \beta_2.
\]
Consequently, $\alpha=\alpha_1$ is a sum of a quadratic algebraic number $(\alpha_1-\alpha_4)/2$ and a cubic algebraic number $-\beta_2$. 

\textit{Sufficiency.} Assume that $\alpha$ is a sum of a quadratic algebraic number $\gamma$ and a cubic algebraic number $\delta$. We will prove that $\alpha$ has degree 6 and some four distinct algebraic conjugates of $\alpha$ satisfy the relation $\alpha_{1}+\alpha_{2}=\alpha_{3}+\alpha_{4}\notin\mathbb{Q}$. Indeed, let $\gamma_1=\gamma$, $\gamma_2$ be the algebraic conjugates of $\gamma$ and let $\delta_1=\delta$, $\delta_2$, $\delta_3$ be the algebraic conjugates of $\delta$. Since the compositum $\mathbb{Q}(\gamma,\delta)$ contains $\gamma$ and $\delta$ of degree 2 and 3, respectively, 
it follows that the degree of  $\mathbb{Q}(\gamma,\delta)$ is divisible by $2\cdot 3=6$. On the other hand, $[\mathbb{Q}(\gamma,\delta):\mathbb{Q}]\leq [\mathbb{Q}(\gamma):\mathbb{Q}]\cdot [\mathbb{Q}(\delta):\mathbb{Q}] = 2\cdot 3=6$. Hence, $[\mathbb{Q}(\gamma,\delta):\mathbb{Q}]=6$. By Lemma~\ref{lempe}, we obtain that $\mathbb{Q}(\gamma,\delta) = \mathbb{Q}(\gamma+\delta)$. Therefore, $\alpha=\gamma+\delta$ is of degree 6. 
Hence, the numbers $\gamma_i+\delta_j$, for $i=1,2$ and $j=1,2,3$, are distinct algebraic conjugates of $\gamma+\delta$. The identity
\[
(\gamma_1+\delta_1) + (\gamma_2+\delta_2) = (\gamma_1+\delta_2) + (\gamma_2+\delta_1)
\]
implies that four distinct algebraic conjugates of $\alpha=\gamma+\delta$ satisfy 
\[
\alpha_{1}+\alpha_{2}=\alpha_{3}+\alpha_{4},
\]
where $\alpha_1=\gamma_1+\delta_1$, $\alpha_2=\gamma_2+\delta_2$, $\alpha_3=\gamma_1+\delta_2$ and $\alpha_4=\gamma_2+\delta_1$. 
Finally, since $tr(\gamma)=\gamma_1+\gamma_2$ and $tr(\delta)=\delta_1+\delta_2+\delta_3$ are rational numbers, we obtain that 
\[
\alpha_1+\alpha_2 = \gamma_1+\delta_1 + \gamma_2+\delta_2 = tr(\gamma) + tr(\delta) - \delta_3
\]
is a cubic algebraic number, and therefore $\alpha_1+\alpha_2\notin\mathbb{Q}$.
\end{proof}

\begin{proof}[Proof of Theorem ~\ref{444}] 
Let $\alpha$ be an algebraic number of degree $d\in\{4,5,6,7\}$  such that $tr(\alpha)=0$. Suppose that four distinct conjugates of $\alpha_{1} := \alpha$ satisfy the relation
\begin{equation}\label{eq:1259}
\alpha_{1}+\alpha_{2}=\alpha_{3}+\alpha_{4}.
\end{equation}
By Lemma $\ref{intro4}$, $d\neq 5$ and $d\neq 7$. Hence, $d=4$ or 6.

$(i)$ Suppose that $d=4$. Then \eqref{eq:1259} together with $tr(\alpha)=0$ imply $\alpha_{1}+\alpha_{2}=\alpha_{3}+\alpha_{4}=0$. Hence, $\alpha_2=-\alpha_1$, and therefore the minimal 
polynomial $p(x)$ of $\alpha$ is of the form $p(x)=x^4+ax^{2}+b\in\mathbb{Q}[x]$.

% Let $\alpha$ be an algebraic number of degree $d=4$ over $\mathbb{Q}$. Suppose that four distinct conjugates of $\alpha_{1} := \alpha$ satisfy the relation
% \[
% \alpha_{1}+\alpha_{2}=\alpha_{3}+\alpha_{4}.
% \]
% %Let $d=4$ and consider the relation $\alpha_{1}+\alpha_{2}=\alpha_{3}+\alpha_{4}$. 
% Notice that for a quartic equation $x^4 + a_{3}x^3 + a_{2}x^2 + a_{1}x + a_{0} = 0$ with $a_{0}, a_{1}, a_{2}, a_{3}\in\mathbb{Q}$ substitution $x = y - \frac{a_{3}}{4}$ gives $y^4+ay^2+by+c=0$ with $a, b, c\in\mathbb{Q}$. So by denoting $\beta = \alpha + \frac{a_{3}}{4}$ and $\beta_{i}=\alpha_{i} + \frac{a_{3}}{4}$ we get that $\beta_{1}+\beta_{2}=\beta_{3}+\beta_{4}$ and $\beta_{1}+\beta_{2}+\beta_{3}+\beta_{4}=0$. These two equations of conjugates imply that $\beta_{1}=-\beta_{2}$ and $\beta_{3}=-\beta_{4}$, thus $-\beta$ is a conjugate of $\beta$. 

% Therefore the minimal polynomial of $\beta$ (of degree 4) over $\mathbb{Q}$ must be of the form $y^4+ay^{2}+b\in\mathbb{Q}[x]$. Since $\frac{a_{3}}{4}$ might be any rational number, we obtain minimal polynomials of the form $p(x+c)$, where $p(x)=x^4+ax^{2}+b\in\mathbb{Q}[x]$.

Conversely, let $p(x)=x^4+ax^{2}+b\in\mathbb{Q}[x]$ be an irreducible polynomial. Let $\beta,\gamma\in\mathbb{C}$ be two distinct roots of $p(x)$ such that $\gamma\neq -\beta$. 
Then $\alpha_1 = \beta$, $\alpha_2 = -\beta$, $\alpha_3 = \gamma$ and $\alpha_4 = -\gamma$ are 
all the roots of $p(x)$ and the relation \eqref{eq:1259} holds.

% if an irreducible polynomial with $a, b, c\in\mathbb{Q}$ is of the form
% \[
% (x+c)^4+a(x+c)^2+b,
% \]
%then by denoting $t=(x+c)^{2}$ we get $t=\frac{-a\pm\sqrt{a^{2}-4b}}{2}$, and since $x=\pm\sqrt{t}-c$, we get the equality $\alpha_{1}+\alpha_{2}=\alpha_{3}+\alpha_{4} = -2c$ for the roots of a given polynomial.
% then, by denoting $t=x+c$, we obtain a polynomial $t^{4}+at^{2}+b$. The roots of this polynomial must satisfy $\beta_{1}=-\beta_{2}$ and $\beta_{3}=-\beta_{4}$. Thus, $\beta_{1}+\beta_{2}=\beta_{3}+\beta_{4}=0$. Since $x=t-c$, we obtain the equality $\alpha_{1}+\alpha_{2}=\alpha_{3}+\alpha_{4} = -2c$ for distinct roots of a polynomial $(x+c)^4+a(x+c)^2+b$.

$(ii)$ Suppose that $d=6$ and the sum $\alpha_{1}+\alpha_{2}$ in \eqref{eq:1259} is a rational number. 
Then, by Proposition $\ref{333}$, $\alpha_{1}+\alpha_{2} = \alpha_{3}+\alpha_{4}=0$. 
This together with $tr(\alpha)=0$ imply $\alpha_2=-\alpha_1$. Therefore, the minimal 
polynomial $p(x)$ of $\alpha$ is of the form $p(x)=x^6+ax^{4}+bx^2+c\in\mathbb{Q}[x]$. 
Similarly, as in case $(ii)$, we see that some four distinct roots of any such irreducible polynomial   satisfy the relation \eqref{eq:1259}.

$(iii)$ Suppose that $d=6$ and the sum $\alpha_{1}+\alpha_{2}$ in \eqref{eq:1259} is not a rational number. 
Then, by Theorem~\ref{propqc}, $\alpha$ is a sum of a quadratic algebraic number $\gamma$ and a cubic algebraic number $\delta$. 
Let $\gamma_1=\gamma$, $\gamma_2$ be the algebraic conjugates of $\gamma$ and let $\delta_1=\delta$, $\delta_2$, $\delta_3$ be the algebraic conjugates of $\delta$. 
Then the numbers $\gamma_i+\delta_j$, for $i=1,2$ and $j=1,2,3$, are the algebraic conjugates of $\alpha=\gamma+\delta$. 
We have that $tr(\alpha)=0$. On the other hand, $tr(\alpha)$ equals the sum of all the numbers $\gamma_i+\delta_j$, for $i=1,2$ and $j=1,2,3$. The later sum equals $3(\gamma_1+\gamma_2)+2(\delta_1+\delta_2+\delta_3)$. 
Hence, $0=tr(\alpha) = 3tr(\gamma)+2tr(\delta)$. Therefore, $tr(\gamma)/2+tr(\delta)/3=0$ and we can represent $\alpha$ as 
\[
\alpha = \left(\gamma-\frac{tr(\gamma)}{2}\right) + \left(\delta-\frac{tr(\delta)}{3}\right).
\]
Note that $\gamma-tr(\gamma)/2$ and $\delta-tr(\delta)/3$ are quadratic and cubic algebraic numbers, respectively, both having trace zero. 
Consequently, without the loss of generality, we can assume that $tr(\gamma)=0$ and $tr(\delta)=0$ in the expression $\alpha=\gamma+\delta$.  
Then the minimal polynomial of $\gamma$ is of the form $x^2-a$ and the minimal polynomial of $\delta$ 
is of the form $R(x)=x^3+bx+c$, where $a,b,c\in\mathbb{Q}$. Moreover, in view of $tr(\gamma)=0$ and $\gamma_1\gamma_2=-a$, we obtain that $\gamma_1=\pm\sqrt{a}$ and $\gamma_2=\mp\sqrt{a}$. Now the minimal polynomial $p(x)$ of $\alpha$ can be expressed as
\begin{align}
p(x) &= \prod_{\substack{i=1,2\\j=1,2,3}} (x-\gamma_i-\delta_j) = \prod_{i=1,2} (x-\gamma_i-\delta_1)(x-\gamma_i-\delta_2)(x-\gamma_i-\delta_3)\nonumber\\
&= R(x-\gamma_1)R(x-\gamma_2) = R(x-\sqrt{a})R(x+\sqrt{a})\nonumber\\
&=x^{6} + (2 b - 3 a) x^{4} + 2 c x^{3} + (3 a^{2} + b^{2}) x^{2} + 2c(3 a + b) x\label{eq:ppol}\\ &\phantom{=}\,\,- a^{3} - 2 a^{2} b - a b^{2} + c^{2}.\nonumber
\end{align}
% Multiplying the last product, yields 
% \begin{equation*}
% \begin{split}
% p(x)=&x^{6} + (2 b - 3 a) x^{4} + 2 c x^{3} + (3 a^{2} + b^{2}) x^{2} + 2c(3 a + b) x\\ &- a^{3} - 2 a^{2} b - a b^{2} + c^{2}.
% \end{split}
% \end{equation*}
Conversely, given an irreducible polynomial $p(x)$ of the form \eqref{eq:ppol}, we can factor it as $p(x)=R(x-\sqrt{a})R(x+\sqrt{a})$, where  $R(x)=x^3+bx+c$. Note that $\sqrt{a}\notin\mathbb{Q}$, since $p(x)$ is irreducible. Consequently, $\sqrt{a}$ is a quadratic algebraic number. Moreover, $R(x)$ is irreducible. Indeed, if $R(x)$ factors as $R(x)=P(x)Q(x)$ with some polynomials $P(x),Q(x)\in\mathbb{Q}[x]$ both of degree $\geq 1$, then 
\[
p(x) = R(x-\sqrt{a})R(x+\sqrt{a}) = P(x-\sqrt{a})P(x+\sqrt{a})Q(x-\sqrt{a})Q(x+\sqrt{a})
\]
with both polynomials $P(x-\sqrt{a})P(x+\sqrt{a})$ and $Q(x-\sqrt{a})Q(x+\sqrt{a})$ having rational coefficients. 
This contradicts the assumption that $p(x)$ is irreducible. Hence, $R(x)$ is irreducible. 
Finally, the factorization $p(x)=R(x-\sqrt{a})R(x+\sqrt{a})$ implies that every root of $p(x)$ is 
a sum of a quadratic algebraic number $\pm\sqrt{a}$ and a root of $R(x)$, which is a cubic algebraic number. 
Therefore, by Theorem~\ref{propqc}, some four distinct roots of $p(x)$ satisfy the relation \eqref{eq:1259} with $\alpha_1+\alpha_2\notin\mathbb{Q}$.
\end{proof}

\textbf{Note.} 
\begin{itemize}
\item[1.] The polynomial in \eqref{eq:ppol} is obtained by expanding the product $R(x-\sqrt{a})R(x+\sqrt{a})$. This can be done either by hand or using a computer algebra system, e.g., SageMath \cite{sagemath}.\\
\item[2.] The polynomial in \eqref{eq:ppol} is irreducible if and only if $a$ is not a square of a rational number and the polynomial $R(x)=x^3+bx+c$ has no roots in $\mathbb{Q}$.
\end{itemize}

\begin{acknowledgement}
\end{acknowledgement}

\printbibliography

\end{document}